\documentclass[journal]{IEEEtran}

\usepackage{amsmath}    
\usepackage{graphics,epsfig,subfigure}    
\usepackage{verbatim}   
\usepackage{color}      
\usepackage{subfigure}  
\usepackage{hyperref}   
\usepackage{amssymb}
\usepackage{epstopdf}
\usepackage{graphicx}
\newcommand{\prob}{Pr}
\newcommand{\expect}[1]{\mathbb{E}\big\{#1\big\}}

\newcommand{\bv}[1]{{\boldsymbol{#1} }}
\newcommand{\script}[1]{{{\cal{#1} }}}

\begin{document}
\title{Utility Optimal Scheduling in Processing Networks}

\author{\large{Longbo Huang, Michael J. Neely}%
\thanks{Longbo Huang (web:  http://www-scf.usc.edu/$\sim$longbohu) and Michael J. Neely (web:  http://www-rcf.usc.edu/$\sim$mjneely)
are with the Dept. of Electrical
Eng., University of Southern California, Los Angeles, CA 90089, USA.}%
\thanks{This material is supported in part  by one or more of 
the following: the DARPA IT-MANET program
grant W911NF-07-0028, the NSF grant OCE 0520324, 
the NSF Career grant CCF-0747525.} }
\maketitle

\newtheorem{rem}{Remark}
\newtheorem{fact_def}{\textbf{Fact}}
\newtheorem{coro}{\textbf{Corollary}}
\newtheorem{lemma}{\textbf{Lemma}}
\newtheorem{main}{\textbf{Proposition}}
\newtheorem{theorem}{\textbf{Theorem}}
\newtheorem{claim}{\emph{Claim}}
\newtheorem{prop}{Proposition}
\newtheorem{assumption}{\textbf{Assumption}}
\newtheorem{condition}{\textbf{Condition}}


 
 \begin{abstract}
 We consider the problem of utility optimal scheduling 
in general \emph{processing networks} with random arrivals and network conditions.  These are generalizations
of traditional data networks where commodities in one or more queues
can be combined to produce new commodities that are delivered to other
parts of the network.   This can be used to model problems
such as in-network data fusion, stream processing, and grid
computing.  Scheduling actions are complicated by the \emph{underflow
problem} that arises when some queues with required components 
go empty.  In this paper, we develop the Perturbed Max-Weight algorithm
(PMW) to achieve optimal utility. The idea of PMW is to perturb
the weights used by the usual Max-Weight algorithm to ``push'' queue
levels towards non-zero values (avoiding underflows). We show that
when the perturbations are carefully chosen, PMW is able to achieve
a utility that is within $O(1/V)$ of the optimal value for any $V\geq1$, while
ensuring an average network backlog of $O(V)$.  
 \end{abstract}


\section{Introduction}
Recently, there has been much attention on developing optimal scheduling algorithms for the class of \emph{processing networks} e.g., \cite{harrison-brownian}, \cite{dai-maxweight-spn},  \cite{jiang-spn}, \cite{zhao-forkandjoin-spn}, \cite{neelyhuang_assembly}. These networks are generalizations of traditional data networks. Contents in these networks can  represent  information, data packets, or certain raw materials, that need to go through multiple processing stages in the network before they can be utilized. One example of such processing networks is the Fork and Join network considered in \cite{zhao-forkandjoin-spn}, which models, e.g., stream processing \cite{amini-streamprocessing} \cite{ibm-stream-processing-10} and grid computing \cite{cao-gridcomuting}.  In the stream processing case, the contents in the network represent different types of data, say voice and video, that need to be combined or jointly compressed, and the network topology represents a particular sequence of operations that needs to be conducted during processing. 
Another example of a processing network is a sensor network that performs data fusion \cite{eswaran-datafusion}, in which case sensor data must first be fused before it is delivered. Finally,  these processing networks also contain the class of manufacturing networks, where raw materials are assembled into products \cite{jiang-spn}, \cite{neelyhuang_assembly}. 

In this paper, we develop optimal scheduling algorithms for the following general utility maximization problem in processing networks. 
We are given a discrete time stochastic processing network. The network state, which describes the network randomness (such as random channel conditions or commodity arrivals), is time varying according to some probability law. A network controller performs some action at every time slot,  based on the observed network state, and subject to the constraint that \emph{the network queues must have enough contents to support the action}. 
The chosen action generates some utility,  but also consumes some amount of contents from some queues, and possibly generates new contents for some other queues.  These contents  cause congestion, and thus lead to backlogs at queues in the network. The goal of the controller is to maximize its time average utility subject to the constraint that the time average total backlog in the network is finite. 


Many of the utility maximization problems in data networks fall into this general framework. For instance,  \cite{eryilmaz_qbsc_ton07},  \cite{neelyenergy},  \cite{huangneelypricing}   \cite{rahulneelycognitive}, \cite{neelysuperfast}, can be viewed as special cases of the above framework which allow scheduling actions to be independent of the content level in the queues (see \cite{yichiang_netopt08} for a survey of problems in data networks). 
By comparing the processing networks with the data networks, we note that the main difficulty in performing utility optimal scheduling in these processing networks is that  \emph{we need to build an optimal scheduling algorithm on top of a mechanism that  prevents queue underflows}. 
Such scheduling problems with underflow constraints are usually  formulated as dynamic programs, e.g., \cite{shuman-underflow}, which require substantial  statistical knowledge of the network randomness, and are usually very difficult to solve. 




In this paper, we develop the Perturbed Max-Weight algorithm (PMW) for achieving optimal utility in processing networks. PMW is a greedy algorithm that makes decisions every time slot, \emph{without requiring any statistical knowledge of the network randomness}. PMW is based on the Max-Weight algorithm developed in the data network context \cite{tassiulas92} \cite{neelynowbook}. There, Max-Weight has been shown to be able to achieve a time average utility that is within $O(1/V)$ of the optimal network utility for any $V\geq1$, while ensuring that the average network delay is $O(V)$, when the network dynamics are i.i.d. \cite{neelynowbook}. 
The idea of PMW is  to perturb the weights used in the Max-Weight algorithm so as to ``push'' the queue sizes towards some nonzero values. Doing so properly, we can ensure that the queues always have enough contents for the scheduling actions. Once this is accomplished, we then do scheduling as in the usual Max-Weight algorithm with the perturbed weights. In this way, we simultaneously avoid queue underflows and achieve good utility performance, and also eliminate the need to solve complex dynamic programs. 

The PMW algorithm is quite different from the approaches used in the processing network literature. \cite{harrison-brownian} analyzes manufacturing networks using Brownian approximations.  \cite{dai-maxweight-spn} applies the Max-Weight algorithm to do scheduling in manufacturing networks, assuming all the queues always have enough contents. \cite{jiang-spn} develops the Deficit Max-Weight algorithm (DMW), by using Max-Weight based on an alternative control process for decision making. \cite{zhao-forkandjoin-spn} formulates the problem as a convex optimization problem to match the input and output rates of the queues, without considering the queueing level dynamics. PMW instead provides a way to explicitly avoid queue underflows, and allow us to compute explicit backlog bounds. Our algorithm is perhaps most similar to the DMW algorithm in \cite{jiang-spn}. DMW achieves the desired performance by bounding the ``deficit'' incurred by the algorithm and applies to both stability  and utility maximization problems. 
Whereas PMW uses perturbations to avoid deficits entirely and allows for more general time varying system dynamics, e.g., random arrivals and random costs. 





The paper is organized as follows: In Section \ref{section:notation} we set up our notations. In Section \ref{section:datafusionexample}, we present a study on a data fusion example to demonstrate the main idea of the paper. 
In Section \ref{section:model} we state the general network model and the scheduling problem. In Section \ref{section:utilitybound} we characterize optimality, 
and in Sections \ref{section:pmw}  we develop the PMW algorithm and show its utility 
can approach the optimum.  Section \ref{section:specificpmw} constructs a PMW algorithm for a more specific yet general network. 
Simulation results are presented in Section \ref{section:simulation}. 


\section{Notations}\label{section:notation}
Here we first set up the notations used in this paper:
$\mathbb{R}$ represents  the set of real numbers. 
$\mathbb{R}_+$ (or $\mathbb{R}_-$) denotes the set of nonnegative (or non-positive) real numbers. 
$\mathbb{R}^n$ (or $\mathbb{R}^n_+$) is  the set of $n$ dimensional \emph{column} vectors, with each element being in $\mathbb{R}$ (or $\mathbb{R}_+$). 
\textbf{Bold} symbols $\bv{a}$ and $\bv{a}^T$ represent a \emph{column} vector and its transpose. 
$\bv{a}\succeq\bv{b}$ means vector $\bv{a}$ is entrywise no less than vector $\bv{b}$. 
$||\bv{a}-\bv{b}||$ is the Euclidean distance of $\bv{a}$ and $\bv{b}$. 
$\bv{0}$ and $\bv{1}$ denote column vectors with all elements being $0$ and $1$. 
For any two vectors $\bv{a}=(a_1, ..., a_n)^T$ and $\bv{b}=(b_1, ..., b_n)^T$, the vector $\bv{a}\otimes\bv{b}=(a_1b_1, ..., a_nb_n)^T$. 
Finally $[a]^+=\max[a, 0]$. 

\section{A data processing example}\label{section:datafusionexample}
In this section, we study  a data fusion example and develop the Perturbed Max-Weight algorithm (PMW) in this case. This example demonstrates the main idea of this paper. 
We will later present our general model in Section \ref{section:model}. 
 
\subsection{Network Settings}
We consider a network shown in Fig. \ref{fig:netexample}, where the network performs a $2$-stage data processing for the data entering into the network. 

\begin{figure}[cht]
\centering
\includegraphics[height=0.8in, width=3.3in]{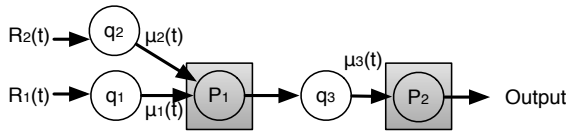}
\caption{An example network consisting of three queues $q_1, q_2, q_3$ and two processors $P_1, P_2$.}\label{fig:netexample}
\end{figure}

In this network, there are two random data streams $R_1(t), R_2(t)$, which represent, e.g., sensed data that come into sensors, or video and voice data that need to be mixed. We assume that $R_i(t)=1$ or $0$, equally likely, for $i=1, 2$. At every time slot, the network controller first decides whether or not to admit the new arrivals, given that accepting any one new arrival unit incurs a cost of $1$. The controller then has to decide how to activate the two processors $P_1, P_2$ for data processing. We assume that both processors can be activated simultaneously. When activated, $P_1$ consumes one unit of data from both $q_1$ and $q_2$, and generates one unit of fused data into $q_3$. This data needs further processing that is done by $P_2$.  When $P_2$ is activated, it consumes one unit of data from $q_3$, and generates one unit of processed data. We assume that each unit of successfully fused and processed data generates a profit of $p(t)$, where $p(t)$ is i.i.d. and takes value $3$ or $1$ with equal probabilities. The network controller's objective is to maximize the average utility, i.e., profit minus cost, subject to queue stability. 

For the ease of presenting the general model later, we define a \emph{network state} $S(t)= (R_1(t), R_2(t))$, \footnote{The network state here contains just   $R_1(t)$ and $R_2(t)$. More complicated settings, where the amount consumed from queues may also depend on the random link conditions between queues and processors can also be modeled by incorporating the link components into the network state, e.g., \cite{huangneely_dr_wiopt09}. } which describes the current network randomness. We also denote the controller's \emph{action} at time $t$ to be $x(t)=(D_1(t), D_2(t), I_1(t), I_2(t))$, where $D_j(t)=1$ ($D_j(t)=0$) means to admit (reject) the new arrivals into queue $j$, and $I_i(t)=1$ ($I_i(t)=0$) means processor $P_i$ is activated (turned off). We note the following \emph{no-underflow constraints} must be met for all time when we activate processors $P_1, P_2$:
\begin{eqnarray}
I_1(t) \leq q_1(t), I_1(t)\leq q_2(t),  I_2(t)  \leq q_3(t).\label{eq:activation-feasible}
\end{eqnarray}
That is, $I_1(t)=1$ only when $q_1$ and $q_2$ are both nonempty, and $I_2(t)=1$ only if $q_3$ is nonempty. 
Note that \cite{jiang-spn} is the first to identify such no-underflow constraints and propose explicit solution to the queue underflow problems for the context of a processing network. 
Subject to (\ref{eq:activation-feasible}), we can then write the amount of arrivals into $q_1, q_2, q_3$, and the service rates of the queues at time $t$ as functions of the network state $S(t)$ and the action $x(t)$, i.e., 
\begin{eqnarray}
A_j(t) &=& A_j(S(t), x(t)) = D_j(t)R_j(t), \,\,j=1, 2,\nonumber\\
A_3(t) &=& A_3(S(t), x(t)) = I_1(t). \label{eq:arrivalrate-example}\\
\mu_j(t) &=& \mu_j(S(t), x(t)) = I_1(t), \,\, j=1, 2,\nonumber\\
\mu_3(t) &=& \mu_3(S(t), x(t)) = I_2(t).  \label{eq:servicerate-example}
\end{eqnarray}
Then we see that the queues evolve according to the following:
\begin{eqnarray}
q_j(t+1) &=& q_j(t) - \mu_j(t) +A_j(t), \,\, j=1, 2, 3,\,\,\forall\,t. \label{eq:queueing-example}
\end{eqnarray}
The \emph{instantaneous utility} is given by:
\begin{eqnarray}
f(t) &=&f(S(t), x(t)) \nonumber\\
&=& p(t)I_2(t)-D_1(t)R_1(t)-D_2(t)R_2(t).  \label{eq:utility-example}
\end{eqnarray}
The goal is to maximize the time average value of $f(t)$ subject to network stability. 

Note that the constraint (\ref{eq:activation-feasible}) greatly complicates the design of an optimal scheduling algorithm. This is because the decision made at time $t$ may affect the queue states in future time slots, which can in turn affect the set of possible actions in the future. In the following, we will develop the Perturbed Max-Weight algorithm (PMW) for this example. The idea of PMW is use the usual Max-Weight algorithm, but to perturb the weights so as to push the queue sizes towards certain nonzero values. By carefully designing the perturbation, we can simultaneously ensure that the queues always have enough data for processing and the achieved utility is close to optimal.

\subsection{The Perturbed Max-Weight algorithm (PMW)}
We now present the construction of the  PMW algorithm for this simple example (this is extended to general network models in Section \ref{section:pmw}). 
To start, we first define a \emph{perturbation vector} $\bv{\theta}=(\theta_1, \theta_2, \theta_3)^T$ and the Lyapunov function  
$L(t)=\frac{1}{2}\sum_{j=1}^3[q_j(t)-\theta_j]^2$. 
We then define the one-slot conditional drift as:
\begin{eqnarray}
\Delta(t) = \expect{L(t+1)-L(t)\left.|\right. \bv{q}(t)}, \label{eq:drift-example}
\end{eqnarray}
where the expectation is taken over the random network state $S(t)$ and the randomness over the actions. 
Using the queueing dynamics (\ref{eq:queueing-example}), it is easy to obtain that:
\begin{eqnarray*}
\Delta(t) \leq B - \sum_{j=1}^{3}\expect{(q_j(t)-\theta_j) [\mu_j(t)-A_j(t)]\left.|\right. \bv{q}(t)}, 
\end{eqnarray*}
where $B=3$. 
Now we use the ``drift-plus-penalty'' approach in \cite{neelynowbook} to design our  algorithm for this problem. To do so, we define a control parameter $V\geq1$, which will affect our utility-backlog tradeoff, and add to both sides the term $-V\expect{f(t)\left.|\right. \bv{q}(t)}$ to get: 
\begin{eqnarray}
\hspace{-.3in}&&\Delta(t) -V\expect{f(t)\left.|\right. \bv{q}(t)}   \label{eq:drift-example}\\
\hspace{-.3in}&&\qquad\qquad   \leq B-V\expect{f(t)\left.|\right. \bv{q}(t)} \nonumber\\
\hspace{-.3in}&&\qquad\qquad\qquad\quad- \sum_{j=1}^{3}\expect{(q_j(t)-\theta_j) [\mu_j(t)-A_j(t)]\left.|\right. \bv{q}(t)} . \nonumber
\end{eqnarray}
Denote $\Delta_V(t)=\Delta(t) -V\expect{f(t)\left.|\right. \bv{q}(t)}$, and plug   (\ref{eq:arrivalrate-example}), (\ref{eq:servicerate-example}) and (\ref{eq:utility-example}) into the above, to get: 
\begin{eqnarray}
\hspace{-.3in}&&\quad\Delta_V(t) \leq B+ \expect{D_1(t)R_1(t)[q_1(t)-\theta_1+V]\left.|\right. \bv{q}(t)}  \label{eq:drift2-example}\\
\hspace{-.3in}&&\qquad\qquad\qquad + \expect{D_2(t)R_2(t)[q_2(t)-\theta_2+V]\left.|\right. \bv{q}(t)}\nonumber\\
\hspace{-.3in}&&\qquad\qquad\qquad- \expect{I_2(t)[q_3(t)-\theta_3+p(t)V]\left.|\right. \bv{q}(t)}\nonumber\\
\hspace{-.3in}&&\qquad - \expect{I_1(t)[q_1(t)-\theta_1+q_2(t)-\theta_2-(q_3(t)-\theta_3)]\left.|\right. \bv{q}(t)}.\nonumber
\end{eqnarray}
We now develop our PMW algorithm by choosing an action at every time slot to \emph{minimize the right-hand side (RHS) of (\ref{eq:drift2-example})} subject to (\ref{eq:activation-feasible}). The algorithm then works as follows:

\underline{\emph{PMW:}} At every time slot, observe $S(t)$ and $\bv{q}(t)$, and do the following:
\begin{enumerate}
\item \underline{Data Admission:} Choose $D_j(t)=1$, i.e., admit the new arrivals to $q_j$ if:
\begin{eqnarray}
q_j(t)-\theta_j+V<0,\,\,j=1,2, \label{eq:admit-example}
\end{eqnarray}
else set $D_j(t)=0$ and reject the arrivals. 
\item \underline{Processor Activation:} Choose $I_1(t)=1$, i.e., activate processor $P_1$, if $q_1(t)\geq1$, $q_2(t)\geq1$, and that:
\begin{eqnarray}
q_1(t)-\theta_1+q_2(t)-\theta_2-(q_3(t)-\theta_3)>0,\label{eq:activateP1-example}
\end{eqnarray}
else choose $I_1(t)=0$. Similarly, choose $I_2(t)=1$, i.e., activate processor $P_2$, if $q_3(t)\geq1$, and that:
\begin{eqnarray}
q_3(t)-\theta_3+p(t)V>0, \label{eq:activateP2-example}
\end{eqnarray}
else choose $I_2(t)=0$. 
\item \underline{Queueing update:} Update $q_j(t),\,\forall\, j$, according to (\ref{eq:queueing-example}). 
\end{enumerate}

\vspace{-.0in}
\subsection{Performance of PMW}
Here we analyze the performance of PMW. We will first 
prove the following important claim: \emph{under a proper $\bv{\theta}$ vector, PMW minimizes the RHS of (\ref{eq:drift2-example}) over all possible policies of arrival admission and processor activation, including those that choose actions regardless of the constraint (\ref{eq:activation-feasible})}. 
We then use this claim to prove the performance of PMW, by comparing the value of the RHS of (\ref{eq:drift2-example}) under PMW versus that under an alternate policy. 


To prove the claim, we first see that the policy that minimizes the RHS of (\ref{eq:drift2-example}) without the constraint (\ref{eq:activation-feasible})  differs from PMW only in the processor activation part, where PMW also considers the constraints $q_1(t)\geq1$, $q_2(t)\geq1$ and $q_3(t)\geq1$. Thus if one can show that these  constraints are indeed redundant in the PMW algorithm under a proper $\bv{\theta}$ vector, i.e., one can activate the processors without considering them but still ensure them, then PMW minimizes the RHS of (\ref{eq:drift2-example}) over all possible policies.  
In the following, we will use the following $\theta_j$ values: 
\begin{eqnarray}
\theta_1=2V,\,\, \theta_2=2V,\,\, \theta_3=3V. \label{eq:thetavalue-example}
\end{eqnarray}

Let us now look at the queue sizes $q_j(t), j=1,2,3$.
From (\ref{eq:activateP2-example}), we see that $P_2$ is activated if and only if: 
\begin{eqnarray}
q_3(t) \geq \theta_3 - p(t)V+1, \quad\text{and}\quad q_3(t)\geq1. \label{eq:q3bound-example}
\end{eqnarray}
Hence $I_2(t)=1$ whenever $q_3(t)\geq\theta_3-V+1$, but $I_2(t)=0$ unless $q_3(t)\geq\theta_3-3V+1$. Since $q_3$ can receive and deliver at most one unit of data at a time, we get: 
\begin{eqnarray}
\theta_3-V+1\geq q_3(t)\geq \theta_3-3V, \,\,\forall\,\,t. \label{eq:q3bound-ex-new}
\end{eqnarray}
Using $\theta_3=3V$, this implies:
\begin{eqnarray}
2V+1\geq q_3(t) \geq 0,\quad\forall\,\,t. \label{eq:q3bound-final}
\end{eqnarray}
This shows that with $\theta_3=3V$, the activations of $P_2$ are always feasible even if we do not consider the constraint $q_3(t)\geq1$. 

We now look at $q_1(t)$ and $q_2(t)$. We see from (\ref{eq:admit-example}) that for $\theta_1, \theta_2\geq V$, we have:
\begin{eqnarray}
q_j(t) \leq \theta_j -V , \,\, j=1,2.  \label{eq:q1q2bound-example}
\end{eqnarray}
Also, using (\ref{eq:activateP1-example}) and (\ref{eq:q3bound-ex-new}), it is easy to see that when $I_1(t)=1$, i.e., when $P_1$ is turned on, we have: 
\begin{eqnarray}
q_1(t)-\theta_1 + q_2(t)-\theta_2 > q_3(t)-\theta_3 \geq -3V. \label{eq:foo}
\end{eqnarray}
Combining (\ref{eq:foo}) with (\ref{eq:q1q2bound-example}), we see that if $I_1(t)=1$, we have: 
\begin{eqnarray}
q_j(t)\geq 1, \,\, j=1, 2. \label{eq:q12bound-final}
\end{eqnarray}
This is so because, e.g., if $q_1(t)=0$, then $q_1(t)-\theta_1 = -\theta_1=-2V$. Since  $q_2(t)-\theta_2\leq -V$ by (\ref{eq:q1q2bound-example}), we thus have: 
\begin{eqnarray*}
q_1(t)-\theta_1 + q_2(t)-\theta_2 \leq -2V-V = -3V,
\end{eqnarray*}
which cannot be greater than $-3V$ in (\ref{eq:foo}). Thus by (\ref{eq:q3bound-final}) and (\ref{eq:q12bound-final}),  we have:  
\begin{eqnarray}
q_j(t)\geq0,\quad j=1, 2, 3,\, \forall\,\, t. \label{eq:nounderflow}
\end{eqnarray}
This shows that by using the $\theta_j$ values in (\ref{eq:thetavalue-example}), PMW automatically ensures that no queue underflow happens, and hence PMW minimizes the RHS of (\ref{eq:drift2-example}) over all possible policies. 

Given the above observation, the utility performance of PMW can now be analyzed as the usual Max-Weight algorithm. 
Specifically, using a similar argument as in \cite{neelyhuang_assembly}, we can  compare the drift under PMW with a stationary randomized algorithm which chooses scheduling actions purely as a function of $S(t)$, and achieves $\expect{\mu_j(t)-A_j(t)\left.|\right.\bv{q}(t)}=0$ for all $j$ and $\expect{f(t)\left.|\right.\bv{q}(t)}=f^*_{av}=\frac{1}{2}$, where $f^*_{av}$ is the optimal average utility. Note that \emph{this comparison will not have been possible here without using the perturbation to ensure (\ref{eq:nounderflow})}. 
Now plugging this  policy into (\ref{eq:drift-example}), we obtain:
\begin{eqnarray}
\hspace{-.3in}&&\Delta(t) -V\expect{f(t)\left.|\right. \bv{q}(t)}    \leq B-Vf^*_{av}. 
\end{eqnarray}
Taking expectations over $\bv{q}(t)$ on both sides and summing it over $t=0, 1, ..., T-1$, we get:
\begin{eqnarray}
\hspace{-.4in}&&\expect{L(T)-L(0)} - V\sum_{t=0}^{T-1}\expect{f(t)}    \leq TB-VTf^*_{av}. 
\end{eqnarray}
Now rearranging the terms, dividing both sides by $VT$,  and using the fact that $L(t)\geq0$, we get:
\begin{eqnarray}
\hspace{-.3in}&&\frac{1}{T}\sum_{t=0}^{T-1}\expect{f(t)}  \geq f^*_{av}-\frac{B}{V} -\frac{\expect{L(0)}}{TV}. 
\end{eqnarray}
Taking a liminf as $T\rightarrow\infty$, and using $\expect{L(0)}<\infty$, we get: 
\begin{eqnarray}
\hspace{-.3in}&&f^{PMW}_{av}=\liminf_{T\rightarrow\infty}\frac{1}{T}\sum_{t=0}^{T-1}\expect{f(t)}  \geq f^*_{av}-\frac{B}{V}, 
\end{eqnarray}
where $f^{PMW}_{av}$ denotes the time average utility achieved by PMW. 
This thus shows that PMW is able to achieve a time average utility that is within $O(1/V)$ of the optimal value, and guarantees $q_j(t)\leq O(V)$ for all time. Note that PMW is similar to the DMW algorithm developed in \cite{jiang-spn}. However,  DMW allows the queues to be empty when activating processors, which may lead to ``deficit,'' whereas PMW effectively avoids this by using a perturbation vector. 


In the following, we will present the general processing network utility optimization model, and  analyze the performance of the general PMW algorithm under this general model. Our analysis uses a duality argument, and will be different  from that in \cite{neelyhuang_assembly}. As we will see, our approach allows one to analyze the algorithm performance without proving the existence of an optimal stationary and randomized algorithm. 





$\vspace{-.2in}$
\section{General System Model}\label{section:model}
In this section, we present the general network model. We consider a network controller that operates a general network with the goal of maximizing the time average utility, subject to the network stability. 
The network is assumed to operate in slotted time, i.e., $t\in\{0,1,2,...\}$. We assume there are $r\geq1$ queues in the network. 

%

$\vspace{-.2in}$
\subsection{Network State}
In every slot $t$, we use $S(t)$ to denote the current network state, which indicates the current network parameters, such as a vector of channel conditions for each link, or a collection of other relevant information about the current network links and arrivals. 
We assume that $S(t)$ is i.i.d. every time slot, with a total of $M$ different random network states denoted by $\script{S} = \{s_1, s_2, \ldots, s_M\}$. \footnote{Note that all our results can easily be extended to the case when $S(t)$ evolves according to a finite state aperiodic and irreducible Markov chain, by using the results developed in \cite{huangneely_qlamarkovian}.} We let $\pi_{s_i}=\prob\{S(t)=s_i\}$.  
The network controller can observe $S(t)$ at the beginning of every slot $t$, but the $\pi_{s_i}$ probabilities are not necessarily known. 

$\vspace{-.22in}$
\subsection{The Utility, Traffic, and Service}\label{subsection:costtrafficservice}
At each time $t$, after observing $S(t)=s_i$ and the network backlog vector, the controller will perform an action $x(t)$. This action represents the aggregate decisions made by the controller at time $t$, which can include, e.g., in the previous example, the set of processors to turn on, or the amount of arriving contents to accept, or both, etc. 

We denote $\script{X}^{(s_i)}$  the set of all feasible actions for network state $s_i$, \emph{assuming all the queues contain enough contents to meet the scheduling requirements}. Note that we always have $x(t)= x^{(s_i)}$ for some $x^{(s_i)}\in\script{X}^{(s_i)}$ whenever  $S(t)=s_i$. The set $\script{X}^{(s_i)}$ is assumed to be time-invariant and compact for all $s_i\in\script{S}$. If the chosen action $x(t)=x^{(s_i)}$ at time $t$ can be performed, i.e., it is feasible and all the queues have enough contents, then the utility, traffic, and service generated by $x(t)$ are as follows:
\begin{enumerate}
\item[(a)] The chosen action has an associated utility given by the utility function $f(t)=f(s_i, x^{(s_i)}): \script{X}^{(s_i)}\mapsto \mathbb{R}$;

\item[(b)] The amount of contents generated by the action to queue $j$ is determined by the traffic function $A_j(t)=A_{j}(s_i, x^{(s_i)}): \script{X}^{(s_i)}\mapsto \mathbb{R}_{+}$, in units of contents; 

\item[(c)] The amount of contents consumed from queue $j$ by the action is given by the rate function $\mu_j(t)=\mu_{j}(s_i, x^{(s_i)}): \script{X}^{(s_i)}\mapsto \mathbb{R}_{+}$, in units of contents;
 \end{enumerate}
Note that $A_j(t)$ includes both the exogenous arrivals from outside the network to queue $j$, and the endogenous arrivals from other queues, i.e., the newly generated contents by processing contents in  some other queues, to queue $j$. We assume the functions $f(s_i, \cdot)$, $\mu_{j}(s_i, \cdot)$ and $A_{j}(s_i, \cdot)$ are continuous, time-invariant, their magnitudes are uniformly upper bounded by some constant $\delta_{max}\in(0,\infty)$ for all $s_i$, $j$, and they are known to the network operator. 



In any actual algorithm implementation, however, we see that not all actions in the set $\script{X}^{(s_i)}$ can be performed when $S(t)=s_i$, due to the fact that some queues may not have enough contents for the action. We say that an action  $x^{(s_i)}\in\script{X}^{(s_i)}$ is feasible at time $t$ with $S(t)=s_i$ only when the following general \emph{no-underflow constraint} is satisfied:
\begin{eqnarray}
q_j(t)\geq \mu_j(s_i, x^{(s_i)}),\,\,\,\forall\,\, j.\label{eq:action-feasible-cond}
\end{eqnarray}
That is, \emph{all the queues must have contents greater than or equal to what will be consumed}. In the following, 
we  assume there exists a set of actions $\{x^{(s_i)}_k\}_{i=1,..., M}^{k=1,2, ..., r+2}$ with $x^{(s_i)}_k\in\script{X}^{(s_i)}$ and some variables $\vartheta^{(s_i)}_k\geq0$ for all $s_i$ and $k$ with $\sum_{k=1}^{r+2}\vartheta^{(s_i)}_k=1$ for all $s_i$, such that:  
\begin{eqnarray}
\sum_{s_i}\pi_{s_i}\big\{\sum_{k=1}^{r+2}\vartheta^{(s_i)}_k[A_{j}(s_i, x^{(s_i)}_k)-\mu_{j}(s_i, x^{(s_i)}_k)]\big\}\leq -\eta,\label{eq:slackness}
\end{eqnarray}
for some $\eta>0$ for all $j$. That is, the ``stability 
constraints'' are feasible with $\eta$-slackness. \footnote{The use of $r+2$ actions here is due to the use of Caratheodory's theorem \cite{bertsekasoptbook} in the proof of Theorem \ref{theorem:optutility}. }
In the following, we use:
\begin{eqnarray}
\bv{A}(t)=(A_1(t), ..., A_r(t))^{T}, \,\,
\bv{\mu}(t)=(\mu_1(t), ..., \mu_r(t))^{T},\label{eq:arrivalservicevector}
\end{eqnarray}
%
to denote the arrival and service vectors at time $t$. 




$\vspace{-.2in}$
\subsection{Queueing, Average Cost, and the Objective }\label{section:queuenotation}
Let $\bv{q}(t)=(q_1(t), ..., q_r(t))^T\in\mathbb{R}^r_{+}$, $t=0, 1, 2, ...$ be the queue backlog vector  process of the network, in units of contents. Due to the feasibility condition (\ref{eq:action-feasible-cond}) of the actions, we see that the queues evolve according to the following dynamics: 
\vspace{-.08in}
\begin{eqnarray}
q_j(t+1)=q_j(t)-\mu_j(t)+A_j(t),\quad\forall j,\,t\geq0,  \label{eq:queuedynamic}
\end{eqnarray}
with some $||\bv{q}(0)||<\infty$. Note that using a nonzero $q_j(0)$ can be viewed as placing an ``initial stock'' in the queues to facilitate algorithm implementation. 
In this paper, we adopt the following notion of queue stability:
\begin{eqnarray}
\overline{q}\triangleq
\limsup_{t\rightarrow\infty}\frac{1}{t}\sum_{\tau=0}^{t-1}\sum_{j=1}^{r}\expect{q_j(\tau)}<\infty.\label{eq:queuestable}
\end{eqnarray}
We also use $f^{\Pi}_{av}$ to denote the time average utility induced by an action-choosing policy $\Pi$, defined as:
\begin{eqnarray}
f^{\Pi}_{av}\triangleq
\liminf_{t\rightarrow\infty}\frac{1}{t}\sum_{\tau=0}^{t-1}\expect{f^{\Pi}(\tau)},\label{eq:timeavcost}
\end{eqnarray}
where $f^{\Pi}(\tau)$ is the utility incurred at time $\tau$ by policy $\Pi$. We call an action-choosing  policy \emph{feasible} if at every time slot $t$ it only chooses actions from the feasible action set $\script{X}^{(S(t))}$ that satisfy (\ref{eq:action-feasible-cond}).  We then call a feasible action-choosing  policy under which (\ref{eq:queuestable}) holds a \emph{stable} policy, and use $f_{av}^*$ to denote the optimal time average utility over all stable policies. 

In every slot, the network controller observes the current network state and the queue backlog vector, and chooses a feasible control action that ensures (\ref{eq:action-feasible-cond}), with the objective of maximizing the time average utility subject to network stability. 
Note that if condition (\ref{eq:action-feasible-cond}) can be ignored, and if any processor only requires contents from a single queue, then this problem falls into the general stochastic network optimization framework considered in \cite{neelynowbook}, in which case it can be solved by using the usual Max-Weight algorithm to achieve a utility that is within $O(1/V)$ of the optimal while ensuring that the average network backlog is $O(V)$.

\vspace{-.0in}
\section{Upper bounding the optimal utility}\label{section:utilitybound}
In this section, we first obtain an upper bound of the optimal utility that the network controller can achieve. This upper bound will later be used to analyze the performance of our algorithm. The result is summarized in the following theorem. 

\begin{theorem} \label{theorem:optutility}
Suppose the initial queue backlog $\bv{q}(t)$ satisfies $\expect{q_j(0)}<\infty$ for all $j=1, ..., r$. Then we have: 
\begin{eqnarray}
Vf_{av}^*\leq \phi^*,
\end{eqnarray}
where $\phi^*$ is the optimal value of the following problem:
\vspace{-.08in}
\begin{eqnarray}
\hspace{-.3in}&& \max: \quad \phi=\sum_{s_i}\pi_{s_i} V\sum_{k=1}^{r+2}a^{(s_i)}_kf(s_i, x^{(s_i)}_k)\label{eq:primal_obj}\\
\hspace{-.3in}&&\quad s.t. \quad \sum_{s_i}\pi_{s_i}\sum_{k=1}^{r+2}a^{(s_i)}_kA_j(s_i, x^{(s_i)}_k)\label{eq:primal_ratecond}\\
\hspace{-.3in}&&  \qquad\qquad\qquad\qquad =  \sum_{s_i}\pi_{s_i}\sum_{k=1}^{r+2}a^{(s_i)}_k\mu_j(s_i, x^{(s_i)}_k)\nonumber\\
\hspace{-.3in}&& \qquad\qquad x^{(s_i)}_k\in\script{X}^{(s_i)},\forall\,s_i, k\label{eq:primal_actioncond}\\
\hspace{-.3in}&& \qquad\qquad a^{(s_i)}_k\geq0, \forall\, s_i, k, \sum_{k}a^{(s_i)}_k=1,\forall\, s_i. \label{eq:primal_probcond}
\end{eqnarray}
\end{theorem}
\begin{proof}
See Appendix A. 
\end{proof}
Note that the problem (\ref{eq:primal_obj}) only requires that the time average input rate into a queue is equal to its time average output rate. This requirement ignores the action feasibility constraint (\ref{eq:action-feasible-cond}), and makes (\ref{eq:primal_obj}) easier to solve than the scheduling problem. 
We now look at the dual problem of the problem (\ref{eq:primal_obj}). The following lemma shows that the dual problem  of (\ref{eq:primal_obj}) does not have to include the variables $\{a^{(s_i)}_k\}_{i=1, ..., M}^{k=1, ..., r+2}$. 
This lemma will also be useful for our later analysis. 
\begin{lemma}\label{lemma:dualproblem}
The dual problem of (\ref{eq:primal_obj}) is given by: 
\begin{eqnarray}
\min: \quad g(\bv{\gamma}),\,\,\, s.t.\quad \bv{\gamma}\in\mathbb{R}^r, \label{eq:dual-problem}
\end{eqnarray}
where the function $g(\bv{\gamma})$ is defined: 
\begin{eqnarray}
\hspace{-.3in}&&g(\bv{\gamma}) = \sup_{x^{(s_i)}\in\script{X}^{(s_i)}}\sum_{s_i}\pi_{s_i}\bigg\{ Vf(s_i, x^{(s_i)}) \label{eq:dual-function}\\
\hspace{-.3in}&&\qquad\qquad\quad -\sum_{j}\gamma_{j}\big[ A_j(s_i, x^{(s_i)})-\mu_j(s_i, x^{(s_i)})\big] \bigg\}. \nonumber
\end{eqnarray}
Moreover, let $\bv{\gamma}^*$ be any optimal solution of  (\ref{eq:dual-problem}), we have:
\begin{eqnarray}
g(\bv{\gamma}^*) \geq \phi^*.\label{eq:nodualitygap}
\end{eqnarray}

\end{lemma}

\begin{proof} (Lemma \ref{lemma:dualproblem}) 
It is easy to see from (\ref{eq:primal_obj}) that the dual function is given by:
\begin{eqnarray}
\hspace{-.3in}&&\hat{g}(\bv{\gamma}) = \sup_{x^{(s_i)}_k, a^{(s_i)}_k}\sum_{s_i}\pi_{s_i}\bigg\{ \sum_{k=1}^{r+2}a^{(s_i)}_kVf(s_i, x^{(s_i)}_k) \label{eq:dual-function-prob}\\
\hspace{-.3in}&&\qquad\quad\quad -\sum_{j}\gamma_{j}\sum_{k=1}^{r+2}a^{(s_i)}_k\big[A_j(s_i, x^{(s_i)}_k)-\mu_j(s_i, x^{(s_i)}_k)\big] \bigg\}. \nonumber
\end{eqnarray}
Due to the use of the $\{a^{(s_i)}_k\}_{i=1, ..., M}^{k=1, ..., r+2}$ variables,  it is easy to see that $\hat{g}(\bv{\gamma})\geq g(\bv{\gamma})$. However, if $\{x^{(s_i)}\}_{i=1}^M$ is a set of maximizers of $g(\bv{\gamma})$, then the set of variables $\{x^{(s_i)}_k, a^{(s_i)}_k\}_{i=1, ..., M}^{k=1, ..., r+2}$ where for each $s_i$,  $x^{(s_i)}_k=x^{(s_i)}$ for all $k$, and $a^{(s_i)}_1=1$ with $a^{(s_i)}_k=0$ for all $k\geq2$, will also be maximizers of $\hat{g}(\bv{\gamma})$. Thus $g(\bv{\gamma})\geq \hat{g}(\bv{\gamma})$. 
This shows that $g(\bv{\gamma})=\hat{g}(\bv{\gamma})$, and hence $g(\bv{\gamma})$ is the dual function of (\ref{eq:primal_obj}). 
(\ref{eq:nodualitygap}) follows from weak duality  \cite{bertsekasoptbook}. 
\end{proof}

In the following, it is useful to define the following function:
\begin{eqnarray}
\hspace{-.3in}&&g_{s_i}(\bv{\gamma}) = \sup_{x^{(s_i)}\in\script{X}^{(s_i)}} \bigg\{ Vf(s_i, x^{(s_i)}) \label{eq:dual-function-state}\\
\hspace{-.3in}&&\qquad\qquad\qquad -\sum_{j}\gamma_{j}\big[ A_j(s_i, x^{(s_i)})-\mu_j(s_i, x^{(s_i)})\big] \bigg\}. \nonumber
\end{eqnarray}
That is, $g_{s_i}(\bv{\gamma})$ is the dual function of (\ref{eq:primal_obj}) when there is a single network state $s_i$. 
We can see from (\ref{eq:dual-function}) and (\ref{eq:dual-function-state}) that:
\begin{eqnarray}
g(\bv{\gamma}) = \sum_{s_i}\pi_{s_i} g_{s_i}(\bv{\gamma}). \label{eq:separable}
\end{eqnarray}
In the following, we will use $\bv{\gamma}^*=(\gamma_1^*, ..., \gamma_r^*)^T$ to denote an optimal solution of the problem (\ref{eq:dual-problem}).

\section{The perturbed max-weight algorithm and its performance}\label{section:pmw}
In this section, we develop the general Perturbed Max-Weight algorithm (PMW) to solve our scheduling problem. 
To start, we first choose a \emph{perturbation vector} $\bv{\theta}=(\theta_1, ..., \theta_r)^T$. Then we 
define the following weighted perturbed Lyapunov function with some positive constants $\{w_j\}_{j=1}^r$:
\begin{eqnarray}
L(t)=\frac{1}{2}\sum_{j=1}^r w_j\big(q_j(t) -\theta_j\big)^2.\label{eq:lyapunov-function}
\end{eqnarray}
We then define the one-slot conditional drift as in (\ref{eq:drift-example}), i.e., 
$\Delta(t)=\expect{L(t+1)-L(t)\left.|\right. \bv{q}(t)}$. 
We will similarly use the ``drift-plus-penalty'' approach in Section \ref{section:datafusionexample} to construct the algorithm. 
Specifically, we first use the queueing dynamic equation (\ref{eq:queuedynamic}), and have the following lemma:
\begin{lemma}\label{lemma:drift-eq}
Under any feasible control policy that can be implemented at time $t$,  we have:
\begin{eqnarray}
\hspace{-.3in}&&\Delta(t)-V\expect{ f(t) \left.|\right. \bv{q}(t)}\leq B -V\expect{ f(t) \left.|\right. \bv{q}(t)}\label{eq:drift}\\
\hspace{-.3in}&&\qquad\qquad\qquad -\sum_{j=1}^rw_j\big(q_j(t)-\theta_j\big)\expect{[\mu_j(t)-A_j(t)]\left.|\right.\bv{q}(t)}, \nonumber
\end{eqnarray}
where $B=\delta^2_{max}\sum_{j=1}^rw_j$. 
\end{lemma}
\begin{proof}
See Appendix B. 
\end{proof}

The general Perturbed Max-Weight algorithm (PMW) is then obtained by choosing an action $x(t)$ from $\script{X}^{(S(t))}$ at  time  $t$ to \emph{minimize the right-hand side (RHS) of (\ref{eq:drift})} subject to (\ref{eq:action-feasible-cond}). 
Specifically, define the function $D_{\bv{\theta}, \bv{q}(t)}^{(s_i)}(x)$ as: 
\begin{eqnarray} 
\hspace{-.3in}&&D_{\bv{\theta}, \bv{q}(t)}^{(s_i)}(x)\label{eq:pmw-action-func}\\
\hspace{-.3in}&&\quad\triangleq Vf(s_i, x)+\sum_{j=1}^{r}w_j\big(q_j(t)-\theta_j\big)\big[\mu_j(s_i, x)-A_j(s_i, x)\big].  \nonumber
\end{eqnarray}
We see that the function $D_{\bv{\theta}, \bv{q}(t)}^{(s_i)}(x)$ is indeed the term inside the conditional expectation on the RHS of (\ref{eq:drift}). We now also define $D^{(s_i)*}_{\bv{\theta}, \bv{q}(t)}$ to be the optimal value of the following problem:
\begin{eqnarray}
\max:\,\, D_{\bv{\theta}, \bv{q}(t)}^{(s_i)}(x), \quad s.t., \,\,\, x^{(s_i)}\in\script{X}^{(s_i)}.\label{eq:QLAeq}
\end{eqnarray}
Hence $D^{(s_i)*}_{\bv{\theta}, \bv{q}(t)}$ is the maximum value of $D^{(s_i)}_{\bv{\theta}, \bv{q}(t)}$ over all possible policies, including those that  may not consider the no-underflow constraint (\ref{eq:action-feasible-cond}). 
The general Perturbed Max-Weight algorithm (PMW) then works as follows: 

\underline{\emph{PMW:}} Initialize the perturbation  vector $\bv{\theta}$. 
At every time slot $t$, observe the current network state $S(t)$ and the backlog $\bv{q}(t)$. If $S(t)=s_i$, choose $x^{(s_i)}\in\script{X}^{(s_i)}$ subject to (\ref{eq:action-feasible-cond})  that makes the value of $D_{\bv{\theta}, \bv{q}(t)}^{(s_i)}(x)$ close to $D^{(s_i)*}_{\bv{\theta}, \bv{q}(t)}$.




Note that depending on the problem structure, the PMW algorithm can usually be implemented easily, e.g., \cite{neelyhuang_assembly}, \cite{neelyenergy}. 
Now we analyze the performance of the PMW algorithm. 
We will prove our result under the following condition: 
\begin{condition}\label{condition:pmw}
There exists some finite constant $C\geq0$, such that 
at every time slot $t$ with a network state $S(t)$,  the value of $D_{\bv{\theta}, \bv{q}(t)}^{(S(t))}(x)$ under PMW  is at least $D_{\bv{\theta}, \bv{q}(t)}^{(S(t))*}-C$.
\end{condition} 

The immediate consequence of Condition \ref{condition:pmw} is that PMW also minimizes the RHS of (\ref{eq:drift}), i.e., the conditional expectation, to within $C$ of its minimum value over all possible policies. 
If $C=0$, then PMW simultaneously ensures (\ref{eq:action-feasible-cond}) and minimizes the RHS of (\ref{eq:drift}), e.g., as in the example in Section \ref{section:datafusionexample}. However, 
we note that Condition \ref{condition:pmw} does not require the value of $D_{\bv{\theta}, \bv{q}(t)}^{(S(t))}(x)$ to be exactly the same as $D_{\bv{\theta}, \bv{q}(t)}^{(S(t))*}$. This allows for more flexibility in constructing the PMW algorithm (See Section \ref{section:specificpmw} for an example). We also note that Condition \ref{condition:pmw} can be ensured, e.g., by carefully choosing the $\theta_j$ values to ensure 
$q_j(t)\geq\delta_{max}$ for all time  \cite{neelyhuang_assembly}. 
We will show  that, under Condition \ref{condition:pmw}, PMW achieves a time average utility that is within $O(1/V)$ of $f^*_{av}$, while guaranteeing that the time average network queue size is $O(V)+\sum_{j}w_j\theta_j$, which is $O(V)$ if $\bv{\theta}=\Theta(V)$ and $w_j=O(1),\,\forall\, j$. 
The following theorem summarizes PMW's performance results. 
\begin{theorem}\label{theorem:pmw1}
Suppose that (\ref{eq:slackness}) holds, that Condition \ref{condition:pmw} holds, and that $\expect{q_j(0)}<\infty$ for all $j=1, ..., r$. Then under PMW, we have: \footnote{Easy to see that (\ref{eq:performance-backlog}) ensures (\ref{eq:queuestable}), hence the network is stable under PMW. }
\begin{eqnarray}
f_{av}^{PMW}&\geq& f_{av}^*-\frac{B+C}{V},\label{eq:performance-utility} \\
\overline{q}^{PMW} &\leq&  \frac{B +C +2V\delta_{max}}{\eta} +\sum_{j=1}^rw_j\theta_j.\label{eq:performance-backlog} 
\end{eqnarray}
Here $B=\delta^2_{max}\sum_{j=1}^rw_j$, $\eta$ is the slackness parameter in Section \ref{subsection:costtrafficservice}, $f_{av}^{PMW}$ is defined in (\ref{eq:timeavcost}) to be the time average expected utility of PMW, and $\overline{q}^{PMW}$ is the time average expected weighted network backlog under PMW, defined:
\begin{eqnarray*}
\overline{q}^{PMW}\triangleq
\limsup_{t\rightarrow\infty}\frac{1}{t}\sum_{\tau=0}^{t-1}\sum_{j=1}^{r}w_j\expect{q_j(\tau)}.
\end{eqnarray*}
\end{theorem}
\begin{proof}
See Appendix C. 
\end{proof}
Theorem \ref{theorem:pmw1} shows that if Condition \ref{condition:pmw} holds,  
then PMW can be used as in previous networking problems, e.g., \cite{neelyenergy}, \cite{huangneelypricing}, to obtain explicit utility-backlog tradeoffs. 
We note that a condition  similar  to Condition \ref{condition:pmw} was assumed in \cite{dai-maxweight-spn}. However,  \cite{dai-maxweight-spn} only considers the usual Max-Weight algorithm, under which case (\ref{eq:action-feasible-cond}) may not be satisfied for all time. Whereas PMW resolves this problem by carefully choosing the perturbation vector. One such example of PMW is the recent work \cite{neelyhuang_assembly}, which applies PMW  to an assembly line scheduling problem and achieves an $[O(1/V), O(V)]$ utility-backlog tradeoff. 

\section{Constructing PMW for networks with output reward}\label{section:specificpmw}
In this section, we look at a specific yet general processing network model, and explicitly construct a PMW algorithm, including finding the proper $\bv{\theta}$ vector and choosing actions at each time slot. 

\subsection{Network Model}
We assume that the network is modeled by an acyclic directed graph $\script{G}=(\script{Q}, \script{P}, \script{L})$. Here $\script{Q}=\script{Q}^{s}\cup\script{Q}^{in}$ is the set of queues, consisting of the set of \emph{source} queues $\script{Q}^{s}$ where arrivals enter the network, and the set of \emph{internal} queues $\script{Q}^{in}$ where contents are stored for further processing. $\script{P}=\script{P}^{in}\cup\script{P}^{o}$ is the set of processors, consisting of a set of \emph{internal} processors $\script{P}^{in}$, which generate partially processed contents for further processing at other processors, and \emph{output} processors $\script{P}^{o}$, which generate fully processed contents and deliver them to the output. $\script{L}$ is the set of directed links that connects $\script{Q}$ and $\script{P}$. Note that a link only exists between a queue in  $\script{Q}$ and a processor in $\script{P}$. We denote $N^{in}_p=|\script{P}^{in}|$, $N^{o}_p=|\script{P}^{o}|$ and $N_p=N^{in}_p+N^{o}_p$. We also denote $N_q^s=|\script{Q}^s|$, $N_q^{in}=|\script{Q}^{in}|$ and $N_q=N^s_q+N^{in}_q$. 

Each processor $P_n$, when activated, consumes a certain amount of contents from a set of \emph{supply} queues, denoted by $\mathbb{Q}_n^S$, and generates some  amount of new contents. These new contents either go to a set of \emph{demand} queues, denoted by $\mathbb{Q}_n^D$, if $P_n\in\script{P}^{in}$, or are delivered to the output if $P_n\in\script{P}^{o}$. 
For any queue $q_j\in\script{Q}$, we use $\mathbb{P}_j^S$ to denote the set of processors that $q_j$ serves as a supply queue, and use $\mathbb{P}_j^D$ to denote the set of processors that $q_j$ serves as a demand queue. 
An example of such a network is shown in Fig. \ref{fig:example-anynet}. 
In the following, we assume that for each processor $P_i\in\script{P}^{in}$, $|\mathbb{Q}_i^D|=1$, i.e., each processor only generates contents for a single demand queue.


We use $\beta_{nj}$ to denote the amount processor $P_n$ consumes from a queue $q_j$ in $\mathbb{Q}_n^S$ when it is activated. For each $P_i\in\script{P}^{in}$, we also use $\alpha_{ih}$ to denote the amount $P_i$ generates into the queue $q_h$ if  $q_h=\mathbb{Q}_i^D$, when it is activated. For a processor $P_k\in\script{P}^{o}$, we use $\alpha_{ko}$ to denote the amount of output generated by it when it is turned on. \footnote{Note that here we only consider binary actions of processors. Our results can also be generalized into the case when there are multiple operation levels under which different amount of contents will be consumed and generated.}
We denote $\beta_{max}=\max_{i,j}\beta_{ij}$, $\beta_{min}=\min_{i, j}\beta_{ij}$ and $\alpha_{max}=\max_{i,j,}[\alpha_{ij}, \alpha_{io}]$. We assume that $\beta_{min}, \beta_{max}, \alpha_{max}>0$. 
We also define $M_p$ to be the maximum number of supply queues that any processor can have, define $M^d_q$ to be the maximum number of processors that any queue can serve as a  demand queue, and define $M^s_q$ to be the maximum number of processors that any queue can serve as a supply queue. 
We use $R_j(t)$ to denote the amount of contents arriving to a source queue $q_j\in\script{Q}^{s}$ at time $t$. We assume $R_j(t)$ is i.i.d. every slot, and that $R_j(t)\leq R_{max}$ for all $q_j\in\script{Q}^{s}$ and all $t$. We assume that there are no exogenous arrivals into the queues in $\script{Q}^{in}$.

\begin{figure}[cht]
\centering
\includegraphics[height=1.4in, width=2.8in]{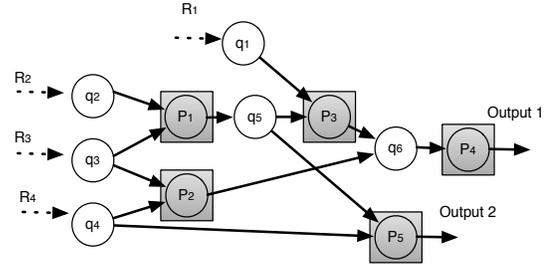}
\caption{A general processing network. }\label{fig:example-anynet}
\end{figure}

We assume that in every slot $t$, admitting any unit amount of $R_j(t)$ arrival incurs a cost of $c_j(t)$, and that activating any internal processor $P_i\in\script{P}^{in}$ incurs a cost of $C_i(t)$, whereas activating any output processor $P_k\in\script{P}^{o}$ generates a profit of $p_k(t)$ per unit output content. 
\footnote{This can be viewed as the difference between profit and cost associated with these processors.} We assume $c_j(t), C_i(t), p_k(t)$ are all i.i.d. every time slot. 
In the following, we also assume that $p_{min}\leq p_k(t)\leq p_{max}$, and that $c_{min}\leq c_j(t)\leq c_{max}$ and $C_{min}\leq C_i(t)\leq C_{max}$ for all $k, j, i$ and for all time. 

Below, we use $I_n(t)=1$ to denote the activation decision of $P_n$, i.e., $I_n(t)=1$ ($I_n(t)=0$) means that $P_n$ is activated (turned off). We also use $D_{j}(t)\in[0, 1]$ to denote the portion of arrivals from $R_j(t)$ that are admitted into $q_j$. 
We assume there exist some general constraint on how the processors can be activated, which can be due to, e.g., resource sharing among processors. We model this constraint by defining an activation vector $\bv{I}(t)=(I_1(t), ..., I_{N_p}(t))$, and then assume that  $\bv{I}(t) \in \script{I}$ for all time, where $\script{I}$ denotes the set of all feasible processor activation decision vectors, assuming all the queues have enough contents for processing. We assume that if  a vector $\bv{I}\in\script{I}$, then by changing one element of $\bv{I}$ from one to zero, the newly obtained vector $\bv{I}'$ satisfies $\bv{I}'\in\script{I}$. Note that the chosen vector $\bv{I}(t)$ must always ensure the constraint (\ref{eq:action-feasible-cond}), which in this case implies that $\bv{I}(t)$ has to satisfy the following constraint:
\begin{eqnarray}
q_j(t)\geq \sum_{n\in\mathbb{P}^S_j} I_n(t)\beta_{nj},\quad\forall\,\,j=1, ..., r. \label{eq:acitvation-cond}
\end{eqnarray}
Under this constraint, we see that the queues evolve according to the following queueing dynamics:
\begin{eqnarray*}
\hspace{-.3in} && q_j(t+1) = q_j(t) -\sum_{n\in\mathbb{P}^S_j} I_n(t)\beta_{nj} + D_j(t)R_j(t),\,\,\forall j\in\script{Q}^{s}, \\
\hspace{-.3in} &&q_j(t+1) = q_j(t)-\sum_{n\in\mathbb{P}^S_j} I_n(t)\beta_{nj} + \sum_{n\in\mathbb{P}^D_j} I_n(t)\alpha_{nj},\,\forall j\in\script{Q}^{in}. 
\end{eqnarray*}
Note that we have used $j\in\script{Q}$ to represent $q_j\in\script{Q}$, and use $n\in\script{P}$ to represent $P_n\in\script{P}$ in the above for notation simplicity. 
The objective is to maximize the time average of the following utility function:
\begin{eqnarray}
\hspace{-.3in}&&f(t) \triangleq \sum_{k\in \script{P}^{o}} I_k(t) p_k(t)\alpha_{ko} - \sum_{j\in\script{Q}^{s}} D_j(t)R_j(t) c_j(t) \label{eq:utility}\\
\hspace{-.3in}&&\qquad\qquad\qquad\qquad\qquad\qquad\qquad\qquad- \sum_{i\in \script{P}^{in}} I_i(t) C_i(t). \nonumber 
\end{eqnarray}
 (\ref{eq:utility}) can be used to model applications where generating completely processed contents is the primary target, e.g., \cite{neelyhuang_assembly}. 

\vspace{-.1in}
\subsection{Relation to the general model}
We see that in this network model, the network state, the action, and the traffic and service functions are given by:
\begin{itemize}
\item The network state is given by:  
\[S(t)=(c_j(t), j\in \script{Q}^{s}, C_i(t), i\in\script{P}^{in}, p_k(t), k\in\script{P}^{o}).\]
\item  The action $x(t) = (D_j(t), j\in\script{Q}^{s}, I_n(t), n\in\script{P})$.
\item  The arrival functions are given by:
\begin{eqnarray*}
\hspace{-.3in}&&A_j(t)=A_j(S(t), x(t)) = D_j(t)R_j(t),\,\,\, \forall\, q_j\in\script{Q}^{s},\\
\hspace{-.3in}&&A_j(t) =A_j(S(t), x(t)) = \sum_{n\in\mathbb{P}^D_j} I_n(t)\alpha_{nj}, \,\,\, \forall\, q_j\in\script{Q}^{in}.
\end{eqnarray*}
\item The service functions are given by:
\begin{eqnarray*}
\mu_j(t) = \mu_j(S(t), x(t)) = \sum_{n\in\mathbb{P}^S_j}I_n(t)\beta_{nj},\quad\forall \,\, j.
\end{eqnarray*}
\end{itemize}
Thus, we see that this network model falls into the general processing network framework in Section \ref{section:model}, and Theorem \ref{theorem:pmw1} will apply in this case. Therefore, in the following, we will construct our PMW algorithm to ensure that Condition \ref{condition:pmw} holds.

\subsection{The PMW algorithm}
We now obtain the PMW algorithm for this general network in the following. We will look for a perturbation vector that is the same in all entries, i.e., $\bv{\theta}=\theta\bv{1}$. 
We first compute the ``drift-plus-penalty'' expression using the weighted perturbed Lyapunov function defined in (\ref{eq:lyapunov-function}) under some given positive constants $\{w_j\}_{j=1}^r$  and some nonzero constant $\theta$: 
\begin{eqnarray}
\hspace{-.3in} && \Delta(t) - V\expect{f(t)\left.|\right. \bv{q}(t)} \leq B\label{eq:drift1-gen}\\
\hspace{-.3in} && - \sum_{ j\in\script{Q}^{s}}\expect{w_j\big[q_j(t)-\theta\big]\big[\sum_{n \in \mathbb{P}^S_j} I_n(t)\beta_{nj}  -R_j(t)D_j(t) \big]\left.|\right. \bv{q}(t)} \nonumber\\
\hspace{-.3in} && - \sum_{j\in\script{Q}^{in}}\expect{w_j\big[q_j(t)-\theta\big]\big[\sum_{n\in \mathbb{P}^S_j} I_n(t)\beta_{nj}  \nonumber \\
\hspace{-.3in} && \qquad\qquad\qquad\qquad\qquad\qquad\qquad- \sum_{n\in \mathbb{P}^D_j} I_n(t)\alpha_{nj}  \big]\left.|\right. \bv{q}(t)} \nonumber\\
\hspace{-.3in} && -V\expect{\sum_{k\in\script{P}^{o}}I_k(t)p_k(t) \alpha_{ko} -\sum_{j\in \script{Q}^{s}}D_j(t)R_j(t)c_j(t)  \nonumber\\
\hspace{-.3in} && \qquad\qquad\qquad\qquad\qquad\qquad\qquad-\sum_{i\in\script{P}^{in}} I_i(t)C_i(t)
\left.|\right. \bv{q}(t)}.\nonumber 
\end{eqnarray}
Here $B= w_{max}\big[\frac{N_q(M^s_q\beta_{max})^2+N^s_qR_{max}^2+N^{in}_q(M^d_q\alpha_{max})^2}{2}\big]$, where $w_{max}=\max_jw_j$. We also denote $w_{min}=\min_jw_j$.  
Rearranging the terms, we get the following:
\begin{eqnarray}
\hspace{-.3in} && \Delta(t) - V\expect{f(t)\left.|\right. \bv{q}(t)} \leq B \label{eq:drift2-gen}\\
\hspace{-.3in} && +\sum_{j\in\script{Q}^{s}} \expect{\big[Vc_j(t)+w_j(q_j(t)-\theta)\big] D_j(t)R_j(t)\left.|\right. \bv{q}(t)} \nonumber\\
\hspace{-.3in} && - \sum_{k\in\script{P}^{o}} \expect{I_k(t) \big[ \sum_{j\in\mathbb{Q}^S_k} w_j(q_j(t)-\theta)\beta_{kj} + Vp_k(t)\alpha_{ko} \big]  \left.|\right. \bv{q}(t)}\nonumber\\
\hspace{-.3in} &&  -\sum_{i\in\script{P}^{in}} \expect{ I_i(t) \big[ \sum_{j\in\mathbb{Q}^S_i}w_j(q_j(t)-\theta) \beta_{ij} - w_h(q_h(t)-\theta)\alpha_{ih} \nonumber\\
\hspace{-.3in} && \qquad \qquad \qquad \qquad \qquad \qquad \qquad \qquad  \quad\,\,\,- VC_i(t)\big]  \left.|\right. \bv{q}(t) }.\nonumber
\end{eqnarray}
Here in the last term $q_h=\mathbb{Q}^D_i$. 
We now present the PMW algorithm. We see that in this case the $D_{\bv{\theta}, \bv{q}(t)}^{(S(t))}(x)$ function is given by:
\begin{eqnarray}
\hspace{-.3in} &&D_{\bv{\theta}, \bv{q}(t)}^{(S(t))}(x) = -\sum_{j\in\script{Q}^{s}} \big[Vc_j(t)+w_j(q_j(t)-\theta)\big] D_j(t)R_j(t) \nonumber\\
\hspace{-.3in} &&\qquad\qquad+\sum_{k\in\script{P}^{o}} I_k(t) \big[ \sum_{j\in\mathbb{Q}^S_k} w_j(q_j(t)-\theta)\beta_{kj} + Vp_k(t)\alpha_{ko} \big] \nonumber\\
\hspace{-.3in} &&\qquad\qquad+\sum_{i\in\script{P}^{in}} I_i(t) \big[ \sum_{j\in\mathbb{Q}^S_i}w_j(q_j(t)-\theta) \beta_{ij} \label{eq:Dfunction}\\
\hspace{-.3in} &&\qquad\qquad\qquad\qquad\qquad\qquad- w_h(q_h(t)-\theta)\alpha_{ih} - VC_i(t)\big]. \nonumber
\end{eqnarray}
Our goal is to design PMW in a way such that under any network state $S(t)$, the value of $D_{\bv{\theta}, \bv{q}(t)}^{(S(t))}(x)$ is close to $D_{\bv{\theta}, \bv{q}(t)}^{(S(t))*}(x)$, which is the maximum value of $D_{\bv{\theta}, \bv{q}(t)}^{(S(t))}(x)$ without the underflow constraint (\ref{eq:acitvation-cond}), i.e., 
\begin{eqnarray*}
D_{\bv{\theta}, \bv{q}(t)}^{(S(t))*}(x) = \max_{D_j(t)\in[0,1], \bv{I}(t)\in\script{I}} D_{\bv{\theta}, \bv{q}(t)}^{(S(t))}(x).
\end{eqnarray*}
Specifically, PMW works as follows:

\underline{\emph{PMW:}} Initialize $\bv{\theta}$. At every time slot $t$, observe $S(t)$ and $\bv{q}(t)$, and do the following:
\begin{enumerate}
\item \underline{Content Admission:} Choose $D_j(t)=1$, i.e., admit all new arrivals to $q_j\in \script{Q}^{s}$ if: 
\vspace{-.05in}
\begin{eqnarray}
Vc_j(t)+ w_j(q_j(t)-\theta) < 0, \label{eq:admit-gen}
\end{eqnarray}
else set $D_j(t)=0$. 

\item \underline{Processor Activation:} For each $P_i\in \script{P}^{in}$, define its weight $W^{(in)}_i(t)$ as:
\begin{eqnarray}
\hspace{-.4in}&&W^{(in)}_i(t)= \big[\sum_{q_j\in \mathbb{Q}^S_i} w_j[q_j(t)-\theta]\beta_{ij} \label{eq:processor-weight-in}\\
\hspace{-.4in}&&\qquad\qquad\qquad\qquad -  w_h[q_h(t)-\theta]\alpha_{ih} - VC_i(t)\big]^+, \nonumber
\end{eqnarray}
where $q_h=\mathbb{Q}^D_i$. 
Similarly, for each $P_k\in\script{P}^{o}$, define its weight $W^{(o)}_k(t)$ as:
\begin{eqnarray}
W^{(o)}_k(t) = \big[\sum_{q_j\in  \mathbb{Q}^S_k} w_j[q_j(t)-\theta]\beta_{kj}  +  Vp_k(t)\alpha_{ko}\big]^+. \label{eq:processor-weight-out}
\end{eqnarray}
Then, choose an activation vector $\bv{I}(t)$ from $\script{I}$ to maximize:
\begin{eqnarray}
\sum_{i\in\script{P}^{in}} I_i(t) W^{(in)}_i(t) + \sum_{k\in\script{P}^{o}} I_k(t) W^{(o)}_k(t), \label{eq:maxweight}
\end{eqnarray}
subject to the following \emph{queue edge constraints: }
\begin{enumerate}
\item For each $P_i\in\script{P}^{in}$, set $I_i(t)=1$, i.e., activate processor $P_i$, only if:  
\begin{itemize}
\item $q_j(t)\geq M_q^s\beta_{max}$ for all $q_j\in \mathbb{Q}^S_i$, 
\item $q_h(t)\leq\theta$, where $q_h= \mathbb{Q}^D_i$.
\end{itemize}
\item For each $P_k\in \script{P}^{o}$, choose $I_k(t)=1$ only if:
\begin{itemize}
\item  $q_j(t)\geq M_q^s\beta_{max}$ for all $q_j\in  \mathbb{Q}^S_k$. 
\end{itemize}
\end{enumerate}
\end{enumerate}
The approach of imposing the queue edge constraints was inspired by the work \cite{neely_universal_scheduling}, where similar constraints are imposed for routing problems. 
Note that if without these queue edge constraints, then PMW will be the same as the action that maximizes $D_{\bv{\theta}, \bv{q}(t)}^{(s_i)}(x)$ without the underflow constraint (\ref{eq:acitvation-cond}). 

\subsection{Performance}
Here we show that PMW indeed ensures that the value of $D_{\bv{\theta}, \bv{q}(t)}^{(S(t))}(x)$ is within some additive constant of $D_{\bv{\theta}, \bv{q}(t)}^{(S(t))*}(x)$. In the following, we assume that:
\begin{eqnarray}
\hspace{-.3in}&&\theta\geq \max\big[\frac{V\alpha_{max}p_{max}}{w_{min}\beta_{min}},  \frac{Vc_{min}}{w_{min}}+M^s_q\beta_{max}\big]. \label{eq:theta-cond-gen} 
\end{eqnarray}
We also assume that the $\{w_j\}_{j=1}^r$ values are chosen such that for any processor $P_i\in\script{P}^{in}$ with the demand queue $q_h$, we have for any supply queue $q_j\in\mathbb{Q}_i^S$ that:
\begin{eqnarray}
w_j\beta_{ij} \geq w_h \alpha_{ih}.\label{eq:w-cond-gen}
\end{eqnarray}
We note that (\ref{eq:theta-cond-gen}) can easily be satisfied and only requires $\theta=\Theta(V)$. A way of choosing the $\{w_j\}_{j=1}^r$ values to satisfy (\ref{eq:w-cond-gen}) is given in  Section \ref{subsection:choosing-w}.  
Note that in the special case when $\beta_{ij}=\alpha_{ij}=1$ for all $i, j$, simply using $w_j=1,\,\forall\, j$ meets the condition (\ref{eq:w-cond-gen}). 

We first look at the queueing bounds. By (\ref{eq:admit-gen}), $q_j$ admits new arrivals only when $q_j(t)<\theta- Vc_{min}/w_j$. Thus: 
\begin{eqnarray}
q_j(t)\leq \theta - Vc_{min}/w_j+R_{max}, \quad\forall\,\,q_j\in\script{Q}^{s}, t. \label{eq:source-q-bound}
\end{eqnarray}
Now by the processor activation rule, we also see that: 
\begin{eqnarray}
0\leq q_j(t) \leq \theta + M_q^d\alpha_{max}, \quad\forall\,\, q_j\in\script{Q}^{in}, t. \label{eq:queuebound-gen}
\end{eqnarray}
This is because under the PMW algorithm, a processor is activated only when all its supply queues have at least $M_q^s\beta_{max}$ units of contents, and when its demand queue has at most $\theta$ units of contents. The first requirement ensures that $q_j(t)\geq 0$ for all time, while the second requirement ensures that $q_j(t)\leq \theta + M_q^d\alpha_{max}$. Below, by defining: 
\begin{eqnarray}
\nu_{max} \triangleq\max\big[M_q^d\alpha_{max}, R_{max}, M^s_q\beta_{max} \big],\label{eq:numax} 
\end{eqnarray}
we can compactly write (\ref{eq:source-q-bound}) and (\ref{eq:queuebound-gen}) as: 
\begin{eqnarray}
0\leq q_j(t) \leq \theta + \nu_{max}, \quad\forall\,\, q_j\in\script{Q}, t. \label{eq:queuebound-gen-nu}
\end{eqnarray}

To prove the performance of the PMW algorithm, it suffices to prove the following lemma, which shows that Condition \ref{condition:pmw} holds for some finite constant $C$ under the PMW algorithm. 
\begin{lemma} \label{lemma:pmw-m-rhs}
Suppose (\ref{eq:theta-cond-gen}) and (\ref{eq:w-cond-gen}) hold. Then under 
PMW,  $D_{\bv{\theta}, \bv{q}(t)}^{(S(t))}(x)\geq D_{\bv{\theta}, \bv{q}(t)}^{(S(t))*}(x)-C$, where $C =  N_pw_{max}M_p\nu_{max}\beta_{max}$. 
\end{lemma}
\begin{proof} 
See Appendix D. 
\end{proof}


We can now directly use Theorem \ref{theorem:pmw1} to have the following corollary concerning the performance of PMW in this case:
\begin{coro}
Suppose (\ref{eq:slackness}), (\ref{eq:theta-cond-gen})  and (\ref{eq:w-cond-gen}) hold. Then PMW achieves the following:
\begin{eqnarray}
f_{av}^{PMW}&\geq& f^*_{av} - \frac{B+C}{V}, \\
\overline{q}^{PMW}&\leq&\frac{B+C+2V\delta_{max}}{\eta} + \theta\sum_{j=1}^rw_j,
\end{eqnarray}
where $C=N_pw_{max}M_p\nu_{max}\beta_{max}$, $f^{PMW}_{av}$ and $\overline{q}^{PMW}$ are the time average expected utility and time average expected weighted backlog under PMW, respectively. $\blacksquare$
\end{coro}

Note that here $\delta_{max}$ can be chosen to be:
\begin{eqnarray*}
\hspace{-.3in}&&\delta_{max}=\max\big[\nu_{max}, N_p^op_{max}\alpha_{max}, \\
\hspace{-.3in}&&\qquad\qquad\qquad\qquad\qquad\qquad N^s_qR_{max}c_{max}+N^{in}_pC_{max}\big]. 
\end{eqnarray*}
Also, since (\ref{eq:theta-cond-gen}) only requires $\theta=\Theta(V)$, and $w_j=\Theta(1)$ for all $j$, we see that PMW indeed achieves an $[O(1/V), O(V)]$ utility-backlog tradeoff in this case.

\subsection{Choosing the $\{w_j\}_{j=1}^r$ values} \label{subsection:choosing-w}
Here we describe how to choose the $\{w_j\}_{j=1}^r$ values to satisfy (\ref{eq:w-cond-gen}). We first let $K$ be the maximum number of processors that any path going from a queue to an output processor can have. It is easy to see that $K \leq |N_p|$ since there is no cycle in the network. 
The following algorithm terminates in $K$ iterations. We use $w_j(k)$ to denote the value of $w_j$ at the $k^{th}$ iteration. In the following, we use $q_{h_n}$ to denote the demand queue of a processor $P_n$. 

\begin{enumerate}
\item At Iteration $1$, denote the set of queues that serve as supply queues for any output processor as $\mathbb{Q}^l_{1}$, i.e., 
\[\mathbb{Q}^l_{1}=\{q_j: \mathbb{P}^S_j\cap\script{P}^{o}\neq\phi\}.\]
Then set $w_j(1)=1$ for each $q_j\in\mathbb{Q}^l_{1}$. Also, set $w_j(1)=0$ for all other $q_j\notin\mathbb{Q}^l_{1}$.
\item At Iteration $k=2, ..., K$, denote $\mathbb{Q}^l_k$ to be the set of queues that serve as supply queues for any processor whose demand queue is in $\mathbb{Q}^l_{k-1}$, i.e., 
\[ \mathbb{Q}^l_{k}=\{q_j: \exists\, P_n\in\mathbb{P}^S_j\,\,\, s.t.\,\,\, \mathbb{Q}^D_n\in\mathbb{Q}^l_{k-1}\}.\]
Then set:
\begin{eqnarray}
w_j(k) = \max\big[w_j(k-1), \max_{n\in\mathbb{P}^S_j} \frac{w_{h_n}(k-1)\alpha_{nh_n}}{\beta_{nj}} \big], \label{eq:w-values}
\end{eqnarray}
where $\alpha_{nh_n}$ is the amount $P_n$ generates into $q_{h_n}$, which is the demand queue of $P_n$. Also, set $w_j(k)=w_j(k-1)$ for all $q_j\notin\mathbb{Q}^l_{k}$.
\item Output the $\{w_j\}_{j=1}^r$ values. 
\end{enumerate}
The following lemma shows that the above algorithm outputs a set of $\{w_j\}_{j=1}^r$ values that satisfy (\ref{eq:w-cond-gen}). 
\begin{lemma} \label{lemma:wvalues}
The $\{w_j\}_{j=1}^r$ values generated by the above algorithm satisfy (\ref{eq:w-cond-gen}).
\end{lemma}
\begin{proof}
See Appendix E. 
\end{proof}


As a concrete example, we consider the example in Fig. \ref{fig:example-anynet}, with the assumption that each processor, when activated, consumes one unit of content from each of its supply queues and generates two units of contents into its demand queue. In this example, we see that $K=3$. Thus the algorithm works as follows:
\begin{enumerate}
\item Iteration $1$, denote $\mathbb{Q}^l_1=\{q_4, q_5, q_6\}$, set $w_4(1)=w_5(1)=w_6(1)=1$. For all other queues, set $w_j(1)=0$.
\item Iteration $2$, denote $\mathbb{Q}^l_2=\{q_1, q_2, q_3, q_4, q_5\}$, set $w_1(2)=w_2(2)=w_3(2)=w_4(2)=w_5(2)=2$. Set  $w_6(2)=1$.
\item Iteration $3$, denote $\mathbb{Q}^l_3=\{q_2, q_3\}$, set $w_2(3)=w_3(3)=4$. Set $w_1(3)=w_4(3)=w_5(3)=2$, $w_6(3)=1$. 
\item Terminate and output $w_1=w_4=w_5=2$, $w_2=w_3=4$, $w_6=1$. 
\end{enumerate}

\vspace{-.1in}
\section{Simulation}\label{section:simulation}
In this section, we simulate the example given in Fig. \ref{fig:example-anynet}. In this example, we assume each $R_j(t)$ is Bernoulli being $0$ or $2$ with equal probabilities. For each $P_i\in\script{P}^{in}$, i.e., $P_1, P_2, P_3$, $C_i(t)$ is assumed to be $1$ or $10$ with probabilities $0.3$ and $0.7$, respectively. 
For the output processors $P_k\in\script{P}^{o}$, i.e., $P_4$ and $P_5$, we assume that $p_k(t)=1$ or $3$ with probabilities $0.6$ and $0.4$, respectively. 
We assume that each processor, when activated, takes one unit of content from each of its supply queues and generates two units of contents into its demand queue (or to the output if it is an output processor). We further assume that all processors can be turned on without affecting others. 
Note that in this case, we have $c_j(t)=0$ for all source queues $q_j$. 

It is easy to see that in this case $M_p=M_q^s=M_q^d=2$,  $\beta_{max}=\beta_{min}=1$, and $\alpha_{max}=2$. 
Using the results in the above, we choose $w_6=1$,  $w_1=w_4=w_5=2$, $w_2=w_3=4$. We also use $\theta=6V$ according to (\ref{eq:theta-cond-gen}). We simulate the PMW algorithm for $V\in\{5, 7, 10, 15,  20, 50, 100\}$. Each simulation is run over $5\times10^6$ slots. 

Fig. \ref{fig:example-sim2} shows the utility and backlog performance of the PMW algorithm. We see that as $V$ increases, the average utility performance quickly converges to the optimal value. The average backlog size also only grows linear in $V$. 
\begin{figure}[cht]
\centering
\includegraphics[height=1.6in, width=3.3in]{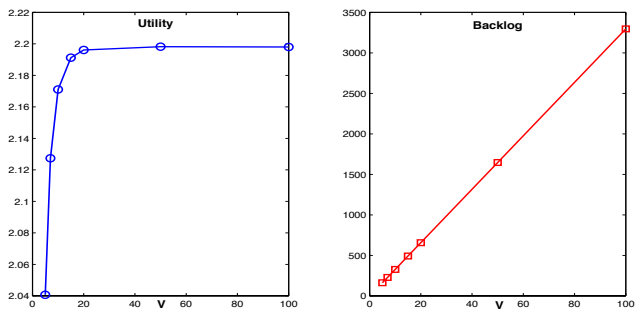}
\caption{Utility and backlog performance of PMW. }\label{fig:example-sim2}
\end{figure}

Fig. \ref{fig:samplepathq} also shows three sample path queue processes  in the first $10^4$ slots under $V=100$. We see that no queue has an underflow. This shows that all the activation decisions of PMW are feasible. It is also easy to verify that the queueing bounds (\ref{eq:source-q-bound}) and (\ref{eq:queuebound-gen}) hold for all time. 
\begin{figure}[cht]
\centering
\includegraphics[height=1.8in, width=3.2in]{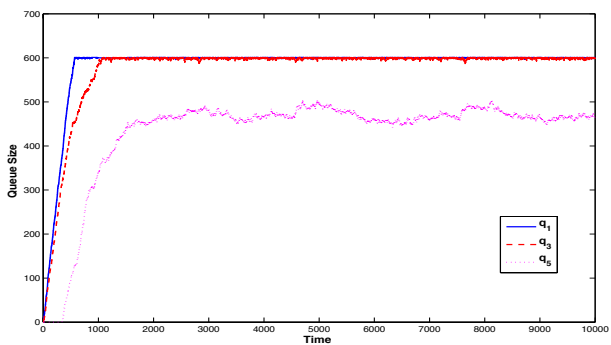}
\caption{Sample path backlog processes with $V=100$. }\label{fig:samplepathq}
\end{figure}

We observe in Fig. \ref{fig:samplepathq}  that the queue sizes usually fluctuate around certain fixed values. Similar ``exponential attraction'' phenomenon has been observed in prior work \cite{huangneely_dr_wiopt09}. Hence our results can also be extended, using the results developed in \cite{huangneely_dr_wiopt09}, to achieve an average utility that is within $O(1/V)$ of the optimal with only $\Theta([\log(V)]^2)$ average backlog size. In this case, we can also implement the PMW algorithm with finite buffers using the idea of \emph{floating queues} in \cite{scott_lifo_ipsn}, which works as follows: For each $q_j$, we associate with it an actual buffer of size $\Theta([\log(V)]^2)$ and a counter. When contents are sent into the queue and the buffer is not full, we store the contents in the actual buffer. However, when the buffer is full and contents are sent to $q_j$, these contents are dropped but the counter is incremented. Whereas if contents are consumed from $q_j$ but $q_j$ does not have enough contents, then the counter is decremented, and the action in that slot is assumed to be null. Under this method, it can be shown  that the  dropping and underflow events happen only with a very small probability. Hence almost all actions are valid. Thus we lose a tiny fraction in the utility performance, but reduce the average backlog size from $O(V)$ to $O([\log(V)]^2)$. 

\section{Conclusion}
In this paper, we develop the Perturbed Max-Weight algorithm (PMW) for utility optimization problems in general processing networks. PMW is based on the usual Max-Weight algorithm for data networks. It has two main functionalities: queue underflow prevention and utility optimal scheduling. PMW simultaneously achieves both objectives by carefully perturbing the weights used in the usual Max-Weight algorithm. We show that PMW is able to achieve an $[O(1/V), O(V)]$ utility-backlog tradeoff. The PMW algorithm developed here can be applied to  problems in the areas of data fusion, stream processing and cloud computing. 

\section*{Appendix A -- Proof of Theorem \ref{theorem:optutility}}
We prove Theorem  \ref{theorem:optutility} in this section, using an argument similar to the one used in \cite{huangneelypricing}. 
\begin{proof} (Theorem \ref{theorem:optutility}) 
Consider any stable scheduling policy $\Pi$, i.e.,  the conditions (\ref{eq:action-feasible-cond})  and (\ref{eq:queuestable}) are satisfied under $\Pi$. We let $\{(f(0), \bv{A}(0), \bv{\mu}(0)), (f(1), \bv{A}(1), \bv{\mu}(1)), ...\}$ be a sequence of (utility, arrival, service) triple generated by $\Pi$. Then there exists a subsequence of times $\{T_i\}_{i=1, 2, ...}$ such that $T_i\rightarrow\infty$ and that the limiting time average utility over times $T_i$ is equal to the liminf average utility under $\Pi$ (defined by (\ref{eq:timeavcost})). 
Now define the conditional average of utility, and arrival minus service over $T$ slots to be:
\begin{eqnarray}
\hspace{-.3in}&&(\phi^{(s_i)}(T); \epsilon_1^{(s_i)}(T); ...; \epsilon_r^{(s_i)}(T)) \triangleq \\
\hspace{-.3in}&&\qquad\qquad\qquad\frac{1}{T} \sum_{t=0}^{T-1} \expect{f(t); \epsilon_1(t); ...; \epsilon_r(t)\left.|\right. S(t)=s_i}, \nonumber
\end{eqnarray}
where $\epsilon_j(t)=A_j(t)-\mu_j(t)$. Using Caratheodory's theorem,  it can be shown, as in \cite{huangneelypricing} that, there exists a set of variables $\{a^{(s_i)}_k(T)\}_{k=1}^{r+2}$ and a set of actions $\{x^{(s_i)}_k(T)\}_{k=1}^{r+2}$ such that:
\begin{eqnarray*}
\hspace{-.3in}&&\phi^{(s_i)}(T) = \sum_{k=1}^{r+2} a^{(s_i)}_k(T)f(s_i, x^{(s_i)}_k(T)),
\end{eqnarray*}
and for all $j=1, ..., r$ that:
\begin{eqnarray*}
\hspace{-.3in}&&\epsilon^{(s_i)}_j(T)= \sum_{k=1}^{r+2} a^{(s_i)}_k(T)[A_j(s_i, x^{(s_i)}_k(T))-\mu_j(s_i, x^{(s_i)}_k(T))]. 
\end{eqnarray*}
Now using the continuity of  $f(s_i, \cdot), A_j(s_i, \cdot), \mu_j(s_i, \cdot)$, and the compactness of all the actions sets $\script{X}^{(s_i)}$, we can thus find a sub-subsequence $\tilde{T}_i\rightarrow\infty$ of $\{T_i\}_{i=1, 2, ...}$ that:
\begin{eqnarray}
\hspace{-.3in}&&a^{(s_i)}_k(\tilde{T}_i)\rightarrow a^{(s_i)}_k, x^{(s_i)}_k(\tilde{T}_i))\rightarrow x^{(s_i)}_k, \\
\hspace{-.3in}&&\phi^{(s_i)}(\tilde{T}_i)\rightarrow\phi^{(s_i)},  \epsilon^{(s_i)}_j(\tilde{T}_i)\rightarrow\epsilon^{(s_i)}_j, \,\forall\, j=1, ..., r. 
\end{eqnarray}
Therefore the time average utility under the policy $\Pi$ can be expressed as: 
\begin{eqnarray}
f^{\Pi}_{av} = \sum_{s_i} \pi_{s_i} \phi^{(s_i)}= \sum_{s_i}\pi_{s_i}\sum_{k=1}^{r+2} a^{(s_i)}_kf(s_i, x^{(s_i)}_k). \label{eq:obj-cara}
\end{eqnarray}
Similarly, the average arrival rate minus the average service rate under $\Pi$ can be written as:
\begin{eqnarray}
\epsilon_j &=& \sum_{s_i} \pi_{s_i}\epsilon_j^{(s_i)} \label{eq:cond-cara}\\
&=& \sum_{s_i} \pi_{s_i}\sum_{k=1}^{r+2} a^{(s_i)}_k [A_j(s_i, x^{(s_i)}_k) - \mu_j(s_i, x^{(s_i)}_k)] \nonumber\\
&\leq& 0. \nonumber
\end{eqnarray}
The last inequality is due to the fact that $\Pi$ is a stable policy and that $\expect{q_j(0)}<\infty$, hence the average arrival rate to any $q_j$ must be no more than the average service rate of the queue \cite{neelythesis}. 
However,  by  (\ref{eq:action-feasible-cond}) we see that what is consumed from a queue is always no more that what is generated into the queue. This implies that the  input rate into a queue is always no less than its output rate. Thus, $\epsilon_j\geq0$ for all $j$. Therefore we conclude that $\epsilon_j =0$ for all $j$. 
Using this fact and (\ref{eq:obj-cara}), we see that $Vf^{\Pi}_{av}\leq\phi^*$, where $\phi^*$ is given in (\ref{eq:primal_obj}). This proves Theorem  \ref{theorem:optutility}. 
\end{proof}

\section*{Appendix B -- Proof of Lemma \ref{lemma:drift-eq}}
Here we prove Lemma \ref{lemma:drift-eq}. 
\begin{proof}
Using the queueing  equation (\ref{eq:queuedynamic}), we have:
\begin{eqnarray*}
&& [q_j(t+1)-\theta_j]^2 \\
&=& [(q_j(t)-\mu_j(t)+A_j(t))-\theta_j]^2\\
&=& [q_j(t)-\theta_j]^2 +(\mu_j(t)-A_j(t))^2 \\
&& \qquad\qquad-2\big(q_j(t)-\theta_j\big) [\mu_j(t)-A_j(t)]\\
&\leq& [q_j(t)-\theta_j]^2 + 2\delta_{max}^2-2\big(q_j(t)-\theta_j\big) [\mu_j(t)-A_j(t)].  
\end{eqnarray*}
Multiplying both sides with $\frac{w_j}{2}$ and summing the above over $j=1, ..., r$, we see that:
\begin{eqnarray*}
L(t+1) - L(t) \leq B -\sum_{j=1}^r w_j\big(q_j(t)-\theta_j\big)[\mu_j(t)-A_j(t)], 
\end{eqnarray*}
where $B=\delta^2_{max}\sum_{j=1}^rw_j$. 
Now add to both sides the term $-Vf(t)$, we get:
\begin{eqnarray}
\hspace{-.2in} &&L(t+1) - L(t)  - Vf(t) \leq B -Vf(t) \label{eq:drift-samplepath}\\
\hspace{-.2in} &&\qquad\qquad\qquad\qquad - \sum_{j=1}^r w_j\big(q_j(t)- \theta_j\big)[\mu_j(t)-A_j(t)]. \nonumber
\end{eqnarray}
Taking expectations over the random network state $S(t)$ on both sides conditioning on $\bv{q}(t)$ proves the lemma. 
\end{proof}

\section*{Appendix C -- Proof of Theorem \ref{theorem:pmw1}}
Here we prove Theorem \ref{theorem:pmw1}. We first have the following simple lemma.
\begin{lemma}
For any network state $s_i$, we have: 
\begin{eqnarray}
D^{(s_i)*}_{\bv{\theta}, \bv{q}(t)}=g_{s_i}((\bv{q}(t)-\bv{\theta})\otimes\bv{w}), \label{eq:drift-dual-eq}
\end{eqnarray}
where $\bv{w}=(w_1, ..., w_r)^T$. 
\end{lemma}
\begin{proof}
By comparing (\ref{eq:QLAeq})  with (\ref{eq:dual-function-state}), we see that the lemma follows. 
\end{proof}

\begin{proof} (Theorem \ref{theorem:pmw1})
We first recall the equation (\ref{eq:drift-samplepath}) as follows:
\vspace{-.06in}
\begin{eqnarray}
\hspace{-.2in} &&L(t+1) - L(t)  - Vf(t) \leq B -Vf(t) \label{eq:drift-samplepath1}\\
\hspace{-.2in} &&\qquad\qquad\qquad\quad - \sum_{j=1}^r w_j\big(q_j(t)- \theta_j\big)[\mu_j(t)-A_j(t)]. \nonumber
\end{eqnarray}
Using $D^{(s_i)}_{\bv{\theta}, \bv{q}(t)}(x)$ defined in (\ref{eq:pmw-action-func}), this can be written as:
\begin{eqnarray}
\hspace{-.2in} &&L(t+1) - L(t)  - Vf(t) \leq  B - D^{(S(t))}_{\bv{\theta}, \bv{q}(t)}(x(t)). \nonumber
\end{eqnarray}
Here $x(t)$ is PMW's action at time $t$. 
According to Condition \ref{condition:pmw}, we see that for any network state $S(t)=s_i$, PMW ensures  (\ref{eq:action-feasible-cond}), and that:
\begin{eqnarray*}
D^{(s_i)}_{\bv{\theta}, \bv{q}(t)}(x) \geq D^{(s_i)*}_{\bv{\theta}, \bv{q}(t)}-C. 
\end{eqnarray*}
%
Using (\ref{eq:drift-dual-eq}), this implies that under PMW, 
\vspace{-.06in}
\begin{eqnarray*}
\hspace{-.4in} &&L(t+1) - L(t)  - Vf(t) \leq B -g_{s_i}((\bv{q}(t)-\bv{\theta})\otimes\bv{w}) + C. 
\end{eqnarray*}
Taking expectations over the random network state on both sides conditioning on $\bv{q}(t)$, and using (\ref{eq:separable}), i.e., $g(\bv{\gamma})=\sum_{s_i}\pi_{s_i}g_{s_i}(\bv{\gamma})$, we get:
\begin{eqnarray}
\hspace{-.55in} &&\Delta(t) -V\expect{f(t)\left.|\right.\bv{q}(t)}\leq B+C -g((\bv{q}(t)-\bv{\theta})\otimes\bv{w}). \label{eq:drift-dual}
\end{eqnarray}
Now using Theorem \ref{theorem:optutility} and Lemma \ref{lemma:dualproblem}, we have: 
\[Vf_{av}^*\leq \phi^* \leq g(\bv{\gamma}^*)\leq g((\bv{q}(t)-\bv{\theta})\otimes\bv{w}).\] 
Therefore,   
\vspace{-.06in}
\begin{eqnarray}
\Delta(t) -V\expect{f(t)\left.|\right.\bv{q}(t)}\leq B+C - Vf_{av}^*. 
\end{eqnarray}
Taking expectations over $\bv{q}(t)$ on both sides and 
summing the above over $t=0, ..., T-1$, we get: 
\vspace{-.06in}
\begin{eqnarray*}
\hspace{-.3in} &&\expect{L(T)-L(0) } -\sum_{t=0}^{T-1}V\expect{f(t) }  \leq T(B+C) -TVf_{av}^*. 
\end{eqnarray*}
Rearranging terms, dividing both sides by $VT$, using the facts that $L(t)\geq0$ and $\expect{L(0)}<\infty$, and taking the liminf as $T\rightarrow\infty$, we get: 
\vspace{-.06in}
\begin{eqnarray}
f^{PMW}_{av}\geq f^*_{av} - (B+C)/V. 
\end{eqnarray}
This proves (\ref{eq:performance-utility}). Now we prove  (\ref{eq:performance-backlog}). 
First, by using the definition of $\hat{g}(\bv{\gamma})$ in (\ref{eq:dual-function-prob}), 
and plugging in the $\{x^{(s_i)}_k, \vartheta^{(s_i)}_k\}_{i=1, ..., M}^{k=1, ..., r+2}$ variables in the $\eta$-slackness assumption (\ref{eq:slackness}) in Section \ref{subsection:costtrafficservice}, we see that: 
\begin{eqnarray}
\hat{g}((\bv{q}(t)-\bv{\theta})\otimes\bv{w}) \geq \eta \sum_{j=1}^rw_j[q_j(t)-\theta_j] - V\delta_{max}. 
\end{eqnarray}
This by Lemma \ref{lemma:dualproblem} implies that: 
\[g((\bv{q}(t)-\bv{\theta})\otimes\bv{w}) \geq \eta \sum_{j=1}^rw_j[q_j(t)-\theta_j]- V\delta_{max}.\]  
Using this in (\ref{eq:drift-dual}), we get: 
 \begin{eqnarray*}
\hspace{-.3in}&&\Delta(t) -V\expect{f(t)\left.|\right.\bv{q}(t)}\leq B+C + V\delta_{max}\\
\hspace{-.3in}&&\qquad\qquad\qquad\qquad\qquad\qquad\qquad\qquad- \eta \sum_{j=1}^rw_j[q_j(t)-\theta_j]. 
\end{eqnarray*}
We can now use a similar argument as above to get: 
\vspace{-.07in}
\begin{eqnarray*}
\hspace{-.3in}&&\eta \sum_{t=0}^{T-1}\sum_{j=1}^rw_j\expect{[q_j(t)-\theta_j]} \\
\hspace{-.3in}&&\qquad\qquad\qquad\leq T(B+C) +2TV\delta_{max} +\expect{L(0)}. 
\end{eqnarray*}
Dividing both sides by $\eta T$ and taking the limsup as $T\rightarrow\infty$, we get:
\vspace{-.16in}
\begin{eqnarray*}
\overline{q}^{PMW} 
& \leq& \frac{B +C +2V\delta_{max}}{\eta} +\sum_{j=1}^rw_j\theta_j. 
\end{eqnarray*}
This completes the proof the theorem. 
\end{proof}

\section*{Appendix D -- Proof of Lemma \ref{lemma:pmw-m-rhs} }
Here we prove Lemma \ref{lemma:pmw-m-rhs} by comparing the values of the three terms in $D_{\bv{\theta}, \bv{q}(t)}^{(S(t))}(x)$ in (\ref{eq:Dfunction}) under PMW versus their values under the action that 
maximizes $D_{\bv{\theta}, \bv{q}(t)}^{(S(t))}(x)$ in (\ref{eq:Dfunction}) subject to only the constraints $D_j(t)\in[0, 1], \,\forall\,j\in\script{Q}^s$ and $\bv{I}(t)\in\script{I}$, called the max-action. That is, under the max-action, $D_{\bv{\theta}, \bv{q}(t)}^{(S(t))}(x)=D_{\bv{\theta}, \bv{q}(t)}^{(S(t))*}(x)$. 
Note that the max-action differs from PMW only in that it does not consider the queue edge constraint. 

 
 
\begin{proof} 
(A) We see that the first term, i.e., $-\sum_{j\in\script{Q}^{s}} \big[Vc_j(t)+ w_j(q_j(t)-\theta)\big] D_j(t)R_j(t)$ is maximized under PMW.
Thus its value is the same as that under the max-action.

(B) We now show that for any processor $P_n\in\script{P}$, if it violates the queue edge constraint, then its weight is bounded by $M_pw_{max}\nu_{max}\beta_{max}$. This will then be used in Part (C)  below  to show that the value of $D_{\bv{\theta}, \bv{q}(t)}^{(S(t))}(x)$ under PMW is within a constant of $D_{\bv{\theta}, \bv{q}(t)}^{(S(t))*}(x)$ under the max-action. 

(B-I) For any $P_i\in\script{P}^{in}$, 
the following are the only two cases under which $P_i$ violates the queue edge constraint. 
\begin{enumerate} 
\item Its demand queue $q_h(t)\geq\theta$. In this case, it is easy to see from (\ref{eq:processor-weight-in}) and (\ref{eq:queuebound-gen-nu}) that:
\begin{eqnarray}
\hspace{-.2in} W^{(in)}_i(t)\leq \sum_{j\in\mathbb{Q}_i^S} w_j\nu_{max}\beta_{ij}\leq  M_pw_{max}\nu_{max}\beta_{max}. \label{eq:weightbound-in0}
\end{eqnarray}
\item One of $P_i$'s supply queue has a queue size less than $M_q^s\beta_{max}$. In this case, we denote $\hat{\mathbb{Q}}_i^S=\{q_j\in\mathbb{Q}_i^S: q_j(t)\geq M_q^s\beta_{max} \}$. Then we see that: 
\begin{eqnarray}
\hspace{-.6in} && W^{(in)}_i(t) = \sum_{j\in\hat{\mathbb{Q}}^S_i } w_j(q_j(t)-\theta)\beta_{ij} - w_h(q_h(t)-\theta)\alpha_{ih}  \nonumber\\
\hspace{-.6in} &&\qquad\qquad\qquad +\sum_{j\in\mathbb{Q}^S_i/ \hat{\mathbb{Q}}^S_i} w_j(q_j(t)-\theta)\beta_{ij}  - VC_i(t)   \nonumber\\
\hspace{-.6in} && \qquad\quad\quad \leq \sum_{j\in\hat{\mathbb{Q}}^S_i } w_j\nu_{max}\beta_{ij}   +w_h\theta \alpha_{ih}   \nonumber\\
\hspace{-.6in} && \qquad\qquad\qquad+ \sum_{j\in\mathbb{Q}^S_i/ \hat{\mathbb{Q}}^S_i}w_j(M^s_q\beta_{max}-\theta)\beta_{ij}. \nonumber
\end{eqnarray}
Here $q_h=\mathbb{Q}^D_i$. 
Now by our selection of $\{w_j\}_{j=1}^r$,  $w_j\beta_{ij}\geq w_h\alpha_{ih}$ for any $q_j\in\mathbb{Q}^S_i$. Also using $\nu_{max}\geq M_q^s\beta_{max}$, we have: 
\begin{eqnarray}
\hspace{-.5in} && W^{(in)}_i(t)  \leq M_pw_{max}\nu_{max}\beta_{max}. \label{eq:weightbound-in}
\end{eqnarray}
\end{enumerate}

(B - II) For any $P_k\in\script{P}^{o}$, we see that it violates the queue edge constraint only when one of its supply queues has size less than $M_q^s\beta_{max}$. In this case, we see that:
\begin{eqnarray*}
\hspace{-.3in} &&W^{(o)}_k(t) \leq \sum_{j\in\hat{\mathbb{Q}}^S_k } w_j(q_j(t)-\theta)\beta_{kj} + Vp_k(t)\alpha_{ko}\nonumber\\
\hspace{-.3in} &&\qquad\qquad\qquad +  \sum_{j\in\mathbb{Q}^S_k/ \hat{\mathbb{Q}}^S_k} w_j(M_q^s\beta_{max}-\theta)\beta_{ij} \nonumber \\
\hspace{-.3in} && \qquad\quad\,\,\, \leq  M_pw_{max}\nu_{max}\beta_{max}  + V\alpha_{max}p_{max} - w_{min}\theta\beta_{min}.\nonumber
\end{eqnarray*}
This by (\ref{eq:theta-cond-gen}) implies that: 
\begin{eqnarray}
\hspace{-.3in} &&  W^{(o)}_k(t)  \leq  M_pw_{max}\nu_{max}\beta_{max}.  \label{eq:weightbound-out}
\end{eqnarray}
Using (\ref{eq:weightbound-in0}), (\ref{eq:weightbound-in}) and (\ref{eq:weightbound-out}), we see that whenever a processor violates the queue edge constraint, its weight is at most $ M_pw_{max}\nu_{max}\beta_{max}$. 

(C) We now show that the value of $D_{\bv{\theta}, \bv{q}(t)}^{(S(t))}(x)$ under PMW satisfies $D_{\bv{\theta}, \bv{q}(t)}^{(S(t))}(x)\geq D_{\bv{\theta}, \bv{q}(t)}^{(S(t))*}(x)-C$, where $C= N_pM_pw_{max}\nu_{max}\beta_{max}$. 

To see this, 
let $\bv{I}^*(t)$ be the activation vector obtained by the max-action, and let $W^*(t)$ be the value of (\ref{eq:maxweight}) under $\bv{I}^*(t)$. We also use $\bv{I}^{PMW}(t)$ and $W^{PMW}(t)$ to denote the activation vector chosen by the PMW algorithm and the value of (\ref{eq:maxweight}) under $\bv{I}^{PMW}(t)$. 
We now construct an alternate activation vector $\hat{\bv{I}}(t)$ by changing all elements in $\bv{I}^*(t)$ corresponding to the processors that violate the queue edge constraints to zero. Note then $\hat{\bv{I}}(t)\in\script{I}$ is a feasible activation vector at time $t$, under which no  processor violates the queue edge constraint. By Part (B) above, we see that the value of (\ref{eq:maxweight}) under $\hat{\bv{I}}(t)$, denoted by $\hat{W}(t)$, satisfies:
\begin{eqnarray*}
\hat{W}(t)\geq W^*(t)-N_pM_pw_{max}\nu_{max}\beta_{max}. 
\end{eqnarray*}
Now since $\bv{I}^{PMW}(t)$ maximizes the value of (\ref{eq:maxweight}) under the queue edge constraints, we have:
\begin{eqnarray*}
W^{PMW}(t)&\geq& \hat{W}(t)\\
&\geq& W^*(t)-N_pw_{max}M_p\nu_{max}\beta_{max}. 
\end{eqnarray*}
Thus, by combining the above and Part (A), we see that PMW maximizes the $D_{\bv{\theta}, \bv{q}(t)}^{(S(t))}(x)$ to within $C=N_pM_pw_{max}\nu_{max}\beta_{max}$ of the maximum. 
\end{proof}

\section*{Appendix E -- Proof of Lemma \ref{lemma:wvalues}}
\begin{proof} (Proof of Lemma \ref{lemma:wvalues}) 
The proof consists of two main steps. In the first step, we show that the algorithm updates each  $w_j$ value at least once. This shows that all the $w_j$ values for all the queues that serve as demand queues are updated at least once. 
In the second step, we show that if $q_h$ is the demand queue of a processor $P_i\in\script{P}^{in}$, then every time after  $w_{h}$ is updated, the algorithm will also update $w_j$ for any $q_j\in\mathbb{Q}^S_i$ before it terminates. This  ensures that (\ref{eq:w-cond-gen}) holds for any $P_i\in\script{P}^{in}$ and hence proves the lemma. 

First we see that after $K$ iterations, we must have $\script{Q}\subset \cup_{\tau=1}^K\mathbb{Q}^l_{\tau}$. This is because at Iteration $k$, we include in $\cup_{\tau=1}^k\mathbb{Q}^l_{\tau}$ all the queues starting from which there exists a path to an output processor that contains $k$ processors. Thus all the $w_j$ values are updated at least once. 

Now consider a queue $q_h$. Suppose $q_h$ is the demand queue of a processor $P_i\in\script{P}^{in}$. We see that there exists a time $\hat{k}\leq K$ at which $w_h$ is last modified. 
Suppose $w_h$ is last modified at Iteration $\hat{k}<K$, in which case $q_h\in\mathbb{Q}^l_{\hat{k}}$. Then all the queues $q_j\in\mathbb{Q}^S_i$ will be in $\mathbb{Q}^l_{\hat{k}+1}$. Thus their $w_j$ values will be modified at Iteration $\hat{k}+1\leq K$. This implies that at Iteration $\hat{k}+1$, we will have $w_j(\hat{k}+1)\beta_{ij}\geq w_h(\hat{k})\alpha_{ih}$.  Since $q_h\notin\mathbb{Q}^l_{k}$ for $ k\geq\hat{k}+1$, we have $w_h(k)=w_h(\hat{k})$ for all $k\geq \hat{k}+1$. Therefore $w_j(k)\beta_{ij}\geq w_h(k)\alpha_{ih}$ $\forall\,\,\hat{k}+1\leq k\leq K$, because $w_j(k)$ is not decreasing. 

Therefore the only case when the algorithm can fail is when $w_h$ is updated at Iteration $k=K$, in which case $w_h$ may increase but the $w_j$ values for $q_j\in\mathbb{Q}^S_i$ are not modified accordingly. However, since $w_h$ is updated at Iteration $k=K$, this implies that there exists a path from $q_h$ to an output processor that has $K$ processors. This in turn implies that starting from any $q_j\in\mathbb{Q}^S_i$, there exists a path to an output processor that contains $K+1$ processors. This contradicts the definition of $K$.  Thus the lemma follows. 
\end{proof}

$\vspace{-.2in}$
\bibliographystyle{unsrt}
\bibliography{../mybib}

\begin{thebibliography}{10}

\bibitem{harrison-brownian}
J.~M. Harison.
\newblock A broader view of brownian networks.
\newblock {\em Ann. Appl. Probab.}, 2003.

\bibitem{dai-maxweight-spn}
J.~G. Dai and W.~Lin.
\newblock Maximum pressure policies in stochastic processing networks.
\newblock {\em Operations Research, Vol 53, 197-218}, 2005.

\bibitem{jiang-spn}
L.~Jiang and J.~Walrand.
\newblock Stable and utility-maximizing scheduling for stochastic processing
  networks.
\newblock {\em Allerton Conference on Communication, Control, and Computing},
  2009.

\bibitem{zhao-forkandjoin-spn}
H.~Zhao, C.~H. Xia, Z.~Liu, and D.~Towsley.
\newblock A unified modeling framework for distributed resource allocation of
  general fork and join processing networks.
\newblock {\em Proc. of ACM Sigmetrics}, 2010.

\bibitem{neelyhuang_assembly}
M.~J. Neely and L.~Huang.
\newblock Dynamic product assembly and inventory control for maximum profit.
\newblock {\em IEEE Conference on Decision and Control (CDC), Atlanta,
  Georgia}, Dec. 2010.

\bibitem{amini-streamprocessing}
L.~Amini, N.~Jain, A.~Sehgal, J.~Silber, and O.~Verscheure.
\newblock Adaptive control of extreme-scale stream processing systems.
\newblock {\em Proc. of International Conference on Distributed Computing
  Systems (ICDCS)}, 2006.

\bibitem{ibm-stream-processing-10}
A.~Biem, E.~Bouillet, H.~Feng, A.~Ranganathan, A.~Riabov, O.~Verscheure,
  H.~Koutsopoulos, and C.~Moran.
\newblock Ibm infosphere streams for scalable, real-time, intelligent
  transportation services.
\newblock {\em Proceedings of the international conference on Management of
  data}, 2010.

\bibitem{cao-gridcomuting}
J.~Cao, S.~A. Jarvis, S.Saini, and G.~R. Nudd.
\newblock Gridflow: workflow management for grid computing.
\newblock {\em Intl. Symposium on Cluster Computing and the Grid (CCGrid)},
  2003.

\bibitem{eswaran-datafusion}
S.~Eswaran, M.~P. Johnson, A.~Misra, and T.~La Porta.
\newblock Adaptive in-network processing for bandwidth and energy constrained
  mission-oriented multi-hop wireless networks.
\newblock {\em Proc. of DCOSS}, 2009.

\bibitem{eryilmaz_qbsc_ton07}
A.~Eryilmaz and R.~Srikant.
\newblock Fair resource allocation in wireless networks using
  queue-length-based scheduling and congestion control.
\newblock {\em IEEE/ACM Trans. Netw.}, 15(6):1333--1344, 2007.

\bibitem{neelyenergy}
M.~J. Neely.
\newblock Energy optimal control for time-varying wireless networks.
\newblock {\em IEEE Transactions on Information Theory 52(7): 2915-2934}, July
  2006.

\bibitem{huangneelypricing}
L.~Huang and M.~J. Neely.
\newblock The optimality of two prices: Maximizing revenue in a stochastic
  network.
\newblock {\em Proc. of 45th Annual Allerton Conference on Communication,
  Control, and Computing (invited paper)}, Sept. 2007.

\bibitem{rahulneelycognitive}
R.~Urgaonkar and M.~J. Neely.
\newblock Opportunistic scheduling with reliability guarantees in cognitive
  radio networks.
\newblock {\em IEEE INFOCOM Proceedings}, April 2008.

\bibitem{neelysuperfast}
M.~J. Neely.
\newblock Super-fast delay tradeoffs for utility optimal fair scheduling in
  wireless networks.
\newblock {\em IEEE Journal on Selected Areas in Communications (JSAC), Special
  Issue on Nonlinear Optimization of Communication Systems}, 24(8), Aug. 2006.

\bibitem{yichiang_netopt08}
Y.~Yi and M.~Chiang.
\newblock Stochastic network utility maximization: A tribute to kelly's paper
  published in this journal a decade ago.
\newblock {\em European Transactions on Telecommunications}, vol. 19, no. 4,
  pp. 421-442, June 2008.

\bibitem{shuman-underflow}
D.~I. Shuman and M.~Liu.
\newblock Energy-efficient transmission scheduling for wireless media streaming
  with strict underflow constraints.
\newblock {\em WiOpt}, 2008.

\bibitem{tassiulas92}
L.~Tassiulas and A.~Ephremides.
\newblock Stability properties of constrained queueing systems and scheduling
  policies for maximum throughput in multihop radio networks.
\newblock {\em IEEE Trans. on Automatic Control, vol. 37, no. 12, pp.
  1936-1949}, Dec. 1992.

\bibitem{neelynowbook}
L.~Georgiadis, M.~J. Neely, and L.~Tassiulas.
\newblock {\em Resource Allocation and Cross-Layer Control in Wireless
  Networks}.
\newblock Foundations and Trends in Networking Vol. 1, no. 1, pp. 1-144, 2006.

\bibitem{huangneely_dr_wiopt09}
L.~Huang and M.~J. Neely.
\newblock Delay reduction via lagrange multipliers in stochastic network
  optimization.
\newblock {\em Proc. of WiOpt, Seoul}, June 2009.

\bibitem{huangneely_qlamarkovian}
L.~Huang and M.~J. Neely.
\newblock Max-weight achieves the exact $[o(1/v), o(v)]$ utility-delay tradeoff
  under markov dynamics.
\newblock {\em arXiv:1008.0200v1}, 2010.

\bibitem{bertsekasoptbook}
D.~P. Bertsekas, A.~Nedic, and A.~E. Ozdaglar.
\newblock {\em Convex Analysis and Optimization}.
\newblock Boston: Athena Scientific, 2003.

\bibitem{neely_universal_scheduling}
M.~J. Neely.
\newblock Universal scheduling for networks with arbitrary traffic, channels,
  and mobility.
\newblock {\em arXiv:1001.0960v1}, Jan 2010.

\bibitem{scott_lifo_ipsn}
S.~Moeller, A.~Sridharan, B.~Krishnamachari, and O.~Gnawali.
\newblock Routing without routes: The backpressure collection protocol.
\newblock {\em 9th ACM/IEEE International Conference on Information Processing
  in Sensor Networks (IPSN)}, 2010.

\bibitem{neelythesis}
M.~J. Neely.
\newblock {\em Dynamic Power Allocation and Routing for Satellite and Wireless
  Networks with Time Varying Channels}.
\newblock PhD thesis, Massachusetts Institute of Technology, Laboratory for
  Information and Decision Systems (LIDS), 2003.

\end{thebibliography}

\end{document}